\documentclass[reqno,11pt]{amsart}
\newcommand{\pa}{\partial}
\usepackage{amsfonts}
\usepackage{amsmath,latexsym,amssymb,amsfonts,amsbsy, amsthm}

\newtheorem{lemma}{Lemma}
[section]
\newtheorem{proposition}{Proposition}
[section]
\newtheorem{theorem}{Theorem}
[section]

\newcommand{\eps}{\varepsilon}

\newcommand{\grad}{\nabla_{\!x}}

\def \del{{\partial}}

\def \DIV{\nabla_{\!x}\!\cdot\! }
\def \LAP{\Delta}

\def \b{\mathrm{b}}

\def \<{\langle}
\def \>{\rangle}

\def \eps{{\varepsilon}}

\begin{document}

\title[Inviscid Limit of the Linearized Compressible Navier-Stokes-Fourier]{Zero viscosity and thermal diffusivity limit of the linearized compressible Navier-Stokes-Fourier equations in the half plane}
\author[Y. Ding]{Yutao Ding}
\address{Mathematical Sciences Center, Jin Chunyuan West Building, Beijing, 100084}
\email{ytding@math.tsinghua.edu.cn}
\author[N. Jiang]{Ning Jiang}
\address{Mathematical Sciences Center, Jin Chunyuan West Building, Beijing, 100084}
\email{njiang@math.tsinghua.edu.cn}


\begin{abstract}
We study the zero viscosity and heat conductivity limit of an initial boundary problem for the linearized Navier-Stokes-Fourier equations of a compressible viscous and heat conducting fluid in the half plane. We consider the case that the viscosity and thermal diffusivity converge to zero at the same order. The approximate solution of the linearized Navier-Stokes-Fourier equations with inner and boundary expansion terms is analyzed formally first by multiscale analysis. Then the pointwise estimates of the error terms of the approximate solution are obtained by energy methods, thus establish the uniform stability for the linearized Navier-Stokes-Fourier equations in the zero viscosity and heat conductivity limit. This work is based on \cite{xy} and generalize the results from isentropic case to the general compressible fluid with thermal diffusive effect. Besides the viscous layer as in \cite{xy}, thermal layer appears and coupled with the viscous layer linearly.
\end{abstract}

\maketitle
\section{Introduction}

The evolution of a compressible viscous heat conducting fluid occupying in the half-plane can be described by the density $\rho(t,x) \geq 0$, the velocity $\mathrm{u}(t,x) \in \mathbb{R}^2$, and the temperature $\theta(t,x) \geq 0$ obeying the following Navier-Stokes-Fourier system of equations:
\begin{equation}\label{NSF}
\begin{split}
 \pa_t \rho + \DIV (\rho  \mathrm{u} )&=0\,, \\
 \pa_t(\rho  \mathrm{u} )+ \DIV  (\rho  \mathrm{u} \otimes \mathrm{u} )+\grad p&= \DIV  \mathbb{S}\,,\\
 \pa_t(\rho  Q( \theta ))+  \DIV (\rho  Q( \theta )u )+\theta p_{\theta}  \DIV  \mathrm{u} &=\mathbb{S}:\nabla_x u - \DIV  q  \,,\\
\end{split}
\end{equation}
with $(t,x) \in \mathbb{R}_+ \times \Omega$. Here $\Omega=\{(x_1, x_2)\in \mathbb{R}^2, x_1>0\}$ with boundary $\Gamma=\pa \Omega =\{(x_1, x_2)\in \mathbb{R}^2, x_1=0\}.$ In the system \eqref{NSF}, $Q\in C^2[0,\infty)$ is a given function such that
$$Q(\theta)=\int_0^{\theta}c_v(z)\,\mathrm{d}z,\quad c_v(z)\geq c_v>0\quad \!\text{for}\quad\!\! z>0\,.$$
The quantities $p$, $\mathbb{S},$ and $q$ are determined in terms of $\rho$, $\grad \mathrm{u}$ and $\theta$ through constitutive equations: the first, Newton's law of viscosity,
$\mathbb{S}(\mathrm{u})=\mu(\nabla_x \mathrm{u}+\nabla_x \mathrm{u}^{\top})+\lambda  \DIV \mathrm{u}\mathbb{I},$ with constant viscosity coefficients $\mu, \lambda$ satisfying $ \mu > 0\,,$ $\xi=\lambda + \mu \geq 0\,.$ The second, the state equation $p=p_e(\rho) + \theta p_\theta(\rho)\,.$ The third, Fourier's law $q= -\kappa \grad \theta$ with the thermal diffusivity $\kappa >0$. More details of the derivation of the system \eqref{NSF} can be found in \cite{Feireisl}.

In this paper, we study the system \eqref{NSF} imposed with the non-slip boundary condition for any $T > 0$,
\begin{equation}\label{NSF nonslip BC}
\mathrm{u} =0,\;\;\;\; \theta = 0\quad\text{on}\quad\! \Gamma\times [0,T]
\end{equation}
and the initial data
\begin{equation}\label{NSF initial data}
(p, \mathrm{u}_1, \mathrm{u}_2, \theta )^{\top}(x,0)=(p_0,\mathrm{u}_{1,0},\mathrm{u}_{2,0}, \theta _0)^{\top}(x), \;\;\;x\in\Omega.
\end{equation}

On the other hand, the motion of an inviscid compressible fluid without thermal diffusivity is governed by the compressible Euler equations, which are obtained by formally taking the viscosity coefficients $\mu, \lambda$ and the thermal diffusivity $\kappa$ as zeros in \eqref{NSF}.
\begin{equation}\label{Euler}
\begin{split}
 \pa_t \rho + \DIV (\rho  \mathrm{u} )&=0\,, \\
 \pa_t(\rho  \mathrm{u} )+ \DIV  (\rho  \mathrm{u} \otimes \mathrm{u} )+\grad p&= 0\,,\\
 \pa_t(\rho  Q( \theta ))+  \DIV (\rho  Q( \theta )u )+\theta p_{\theta}  \DIV  \mathrm{u} &=0 \,.\\
\end{split}
\end{equation}
The boundary condition imposed on the compressible Euler equations \eqref{Euler} is
\begin{equation}\label{Euler BC}
 \mathrm{u}_1 = 0 \quad \mbox{on}\quad\! \Gamma \times [0,T]\,.
\end{equation}
We impose the same initial data for \eqref{Euler} as in \eqref{NSF}, i.e. \eqref{NSF initial data}.

In \cite{xy}, Xin and Yanagisawa studied the zero viscosity limit of the linearized Navier-Stokes equations for an isentropic compressible viscous fluid in the half plane. In other words, they considered the equations \eqref{NSF} without the energy equation, and the pressure $p=p(\rho)$ depending only on the density $\rho$. Under the assumption that the coefficients of viscosity $\mu$ and the bulk viscosity $\xi = \mu + \lambda$ have the same order $\eps^2$, they investigated the asymptotic behavior of the solution of an Dirichlet boundary value problem of the linearized Navier-Stokes equation as the parameter $\eps$ goes to zero. It is well-known that due to the disparity of the boundary conditions between Navier-Stokes and Euler equations, a thin region, the so-called boundary layer comes out near the boundary $\Gamma$ in which the values of the unknown functions change drastically in this zero viscosity limit.

In \cite{xy},  by clarifying the special structure of the boundary matrix of the Euler part of the linearized Navier-Stokes equations, Xin and Yanagisawa introduced the boundary characteristic variables and used the asymptotic analysis with multiple length scales to construct an approximate solution of the initial boundary value problem of the linearized Navier-Stokes equations, which included the inner and boundary layer terms. The first order term of the inner expansion is determined by the solution of the linearized Euler equations, i.e. acoustic system, while the the terms in the boundary expansion are solutions of a family of ODEs and Prandtl-type equations. Next, they used the energy method to show the pointwise error estimates of the approximate solution with respect to the viscosity, and derived the uniform stability results of the linearized Navier-Stokes solutions in the zero-viscosity limit.

The present work could be considered as a follow up of \cite{xy}. We also only study the linearized problem of the Navier-Stokes-Fourier equations for a compressible viscous fluid with thermal diffusivity. The major difficulty in the research of the boundary layers of the original nonlinear problem lies in the fact that the leading boundary layer terms satisfy the nonlinear Prandtl-type equations, for which even the local in time existence and regularity in usual Sobolev spaces are wide open problems so far. The only available zero-viscosity limit result for analytic solutions of the nonlinear incompressible Navier-Stokes equations in half-space is due to Sammartino and Caflisch \cite{SC-1}, \cite{SC-2}. The analogue of the same type results for the compressible Navier-Stokes equations is not known.

The purpose of the current paper aims to generalize Xin-Yanagisawa's result by adding the energy (or equivalently temperature) equation. We make the assumption that the thermal diffusivity $\kappa$ is proportional to $\eps^2$, the same order with viscosities. This assumption is for the simplicity of analyzing the structure of viscosity and thermal boundary layers. Physically, the viscosity and thermal diffusivity could have different order. For example, the viscosity is of order $\eps^2$ while the thermal diffusivity is of order $\eps^\gamma$ for some $\gamma > 0$. For this more physical case, even at the formal level, the viscous and thermal layers are not clear so far.

We would like to remark that the case that the viscosity and thermal diffusivity have the same order $\eps^2$ is physically meaningful and interesting. If we start from mesoscopic level of gas dynamics, say, the Boltzmann equation and set the Knudsen number as $\eps^2$. It can be derived formally under the assumption that the deviation from the global Maxwellian with size much smaller than the Knudsen number, the leading fluid dynamics is exactly the linearized Navier-Stokes-Fourier equations for ideal gas with the viscosity and thermal diffusivity of the same order $\eps^2$. In the work under preparation by the second author and N. Masmoudi \cite{J-M}, the acoustic dynamics of the linearized Boltzmann equation with Maxwell reflection boundary condition in half space is studied. The limiting process considered there includes the zero viscosity and thermal diffusivity limit in the current paper (for ideal gas).

For the linearized Euler equations, i.e. the acoustic system, the boundary condition is \eqref{Euler BC}, and no boundary condition for $\theta$. The disparity between the boundary conditions for the linearized Navier-Stokes-Fourier equations, \eqref{NSF nonslip BC} and the condition \eqref{Euler BC} suggest that during the limit $\eps\rightarrow 0$, both viscous and thermal boundary layers are generated.

In this paper, we employ the strategy used in \cite{xy}. We first construct an approximate solutions of the linearized Navier-Stokes-Fourier equations with non-slip boundary condition which includes inner and boundary layer terms. Then using the energy method, we established the pointwise estimates for the error terms, thus derive the uniform stability results for the linearized Navier-Stokes-Fourier solutions in the zero viscosity and thermal diffusivity limits. Our result in this paper is a preliminary consideration of the coupling thermal and viscous layers of the a compressible fluid with both viscosity and thermal diffusivity effects.

Comparing with \cite{xy}, the main novelty of this work is the appearance of the thermal layer. The key point is that because the viscosity and the thermal diffusivity have the same order, the viscous layer and thermal layer appear also appear at the same order, and more importantly, these two types of layers are coupled {\em linearly }. Technically, this fact is reflected by that in \cite{xy}, the viscous layer was described by a single Prandtl-type equation, while in the current work, the viscous and thermal layers are described by a linear system of two Prandtl-type equations. More importantly, the coupling of viscous and thermal layers are linear and weak in the sense that the coupling of the system is only on the unknown functions themselves, but not on their derivatives. This fact also make the analysis easier. We believe that for the case that the viscosity if of order $\eps^2$, while the thermal diffusivity is of order $\eps^\gamma$ with $0< \gamma \neq 2$, these two layers will strongly coupled together. We plan to study this interesting and challenging problem in the near future.

The organization of the paper is as follows: In the rest of this section, we introduce the setting of the problem and state the main theorem. In Section 3, using the method of multiple scales, the approximate solution of initial boundary problem of the linearized Navier-Stokes-Fourier equations \eqref{linearized NSF} is constructed. Section 3 is devoted to the estimates of the error term by energy method. In the last section, we collect some know results and prove the existence of the linear system of Prandtl-type equation.

\subsection{The Setting of the Problem and the Main Result} Now, we first set up the linearized problem of \eqref{NSF}-\eqref{NSF initial data}. Let $ V=(\rho,\mathrm{u}_1,\mathrm{u}_2, \theta )^{\top}$, we rewrite the \eqref{NSF} as the following symmetric form
\begin{equation}\label{symmetric NSF}
 A_0(V)\pa_t V+\sum_{j=1}^2A_j(V)\pa_j V= L V,
\end{equation}
where
\begin{equation*}
\begin{split}
 A_0(V)&=
 \begin{pmatrix}
 \frac{1}{\rho }&0&0&0\\
 0&\frac{\rho }{p_\rho}&0&0\\
 0&0&\frac{\rho }{p_\rho}&0\\
0&0&0&\beta
 \end{pmatrix}\,,\quad
 A_1(V)=
 \begin{pmatrix}
 \frac{\mathrm{u}_1}{\rho }&1&0&0\\
 1&\frac{\rho \mathrm{u}_1}{p_\rho}&0&\frac{p_\theta}{p_\rho}\\
 0&0&\frac{\rho \mathrm{u}_1}{p_\rho}&0\\
0&\frac{p_\theta}{p_\rho}&0&\beta \mathrm{u}_1
 \end{pmatrix}\,,\\
\end{split}
\end{equation*}
and
\begin{equation*}
\begin{split}
 A_2(V)&=
 \begin{pmatrix}
 \frac{\mathrm{u}_2}{\rho }&0&1&0\\
 0&\frac{\rho \mathrm{u}_2}{p_\rho}&0&0\\
 1&0&\frac{\rho \mathrm{u}_2}{p_\rho}&\frac{p_\theta}{p_\rho}\\
0&0&\frac{p_\theta}{p_\rho}&\beta  \mathrm{u}_2
 \end{pmatrix}\,, \quad
L V=
 \begin{pmatrix}
 0\\
\tfrac{\mu}{p_\rho} \LAP \mathrm{u}_1+ \tfrac{\xi}{p_\rho}(\pa_{11}\mathrm{u}_1+\pa_{12}\mathrm{u}_2)\\
\tfrac{\mu}{p_\rho}\LAP \mathrm{u}_2+\tfrac{\xi}{p_\rho}(\pa_{21}\mathrm{u}_1+\pa_{22}\mathrm{u}_2)\\
\tfrac{\kappa}{ \theta p_\rho} \LAP \theta + \tfrac{1}{\theta p_\rho}\mathbb{S}:\nabla_x \mathrm{u} )
 \end{pmatrix}
\end{split}
\end{equation*}
where $\beta=\frac{\rho c_v( \theta )}{ \theta p_\rho}$, $p_\rho\,, p_\theta$ denote the partial derivatives $\frac{\del p}{\del \rho}\,, \frac{\del p}{\del \theta}$, and $\partial_{ij}$ denotes $\partial_{x_i x_j}$.

Let $\mu=\overline{\mu}\varepsilon^2,\xi= \overline{\xi}\varepsilon^2= (\overline{\mu}+ \overline{\lambda})\eps^2$ and $\kappa= \overline{\kappa}\varepsilon^2 $
where $\varepsilon$ is a small positive parameter, the constants $\overline{\mu}, \overline{\lambda}, \overline{\kappa}$ are of order $1$ and independent of $\varepsilon$. We linearize equations \eqref{NSF} around smooth functions $V'(t)=(\rho',\mathrm{u}_1',\mathrm{u}_2',\theta')$ which is a solution to the equations
\eqref{NSF} for $t\in [0,T]$ for some $T > 0$. Then the linearized equations of \eqref{NSF} can be written as equations for $V^\eps=(\rho^\eps,v^\eps_1,v^\eps_2,\theta^\eps)^{\top}$:
\begin{equation}\label{linearized NSF}
\begin{aligned}
 A_0(V')\pa_t V^\eps +\sum_{j=1}^2A_j(V')\pa_j V^\eps &= D_\eps V^\eps \quad \mbox{in}\quad\! \Omega \times [0,T]\,,\\
M^+ V^\eps & = 0 \quad \mbox{on}\quad\! \Gamma \times [0,T]\,, \\
V^\eps(x,0)&=(\rho_0,v_{1,0},v_{2,0},\theta_0)^{\top}(x)= V_0 (x)\quad \mbox{for}\quad\! x\in \Omega\,,
\end{aligned}
\end{equation}
where
\begin{equation}\nonumber
\begin{aligned}
  D_\eps V = \eps^2\begin{pmatrix}
                0\\
                \tfrac{\overline{\mu}}{p'_\rho}\LAP v_1 + \tfrac{\overline{\xi}}{p'_\rho}(\partial_{11}v_1 + \partial_{12}v_2)\\
                \tfrac{\overline{\mu}}{p'_\rho}\LAP v_2 + \tfrac{\overline{\xi}}{p'_\rho}(\partial_{21}v_1 + \partial_{22}v_2)\\
                \tfrac{\overline{\kappa}}{\theta' p'_\rho} \LAP \theta + I(V)
             \end{pmatrix}\,,\quad  M^+ = \begin{pmatrix}
          0&1&0&0\\
          0&0&1&0\\
          0&0&0&1
       \end{pmatrix}\,,
\end{aligned}
\end{equation}
and
\begin{equation}\nonumber
  I(V)=  \tfrac{2\overline{\mu}}{\theta' p'_\rho}(\partial_i\mathrm{u}'_j + \partial_j\mathrm{u}'_i)\partial_i v_j + \tfrac{2\overline{\lambda}}{\theta' p'_\rho}(\mathrm{div}\mathrm{u}') (\mathrm{div} v)\,.
\end{equation}
We can rewrite $D_\eps V$ as matrices form:
\begin{equation}\nonumber
D_\eps V = \eps^2 \begin{pmatrix}
                     0&0&0&0 \\
                     0&\tfrac{\overline{\mu}}{p'_{\rho}}&0&0 \\
                     0&0&\tfrac{\overline{\mu}}{p'_{\rho}}&0 \\
                     0&0&0& \tfrac{\overline{\kappa}}{\theta' p'_{\rho}}
                  \end{pmatrix}\Delta V + \eps^2 \tfrac{\overline{\xi}}{p'_{\rho}}
                  \begin{pmatrix}
                     0&0&0&0 \\
                     0&\partial_{11}&\partial_{12}&0 \\
                     0&\partial_{21}&\partial_{22}&0 \\
                     0&0&0&0
                  \end{pmatrix}V
                  + \eps^2 \sum^2_{j=1}I_j \partial_j V \,,
\end{equation}
where
\begin{equation}
 I_1 = \begin{pmatrix}
           0&0&0&0 \\
           0&0&0&0 \\
           0&0&0&0 \\
           0& 2\tfrac{\overline{\mu}}{\theta' p'_\rho}\partial_1 \mathrm{u}'_1+ \tfrac{\overline{\lambda}}{\theta' p'_\rho}\mathrm{div}\mathrm{u}'& \tfrac{\overline{\mu}}{\theta' p'_\rho}(\partial_1 \mathrm{u}'_2+\partial_2 \mathrm{u}'_1)& 0
       \end{pmatrix}\,,
\end{equation}
and
\begin{equation}
 I_2 = \begin{pmatrix}
           0&0&0&0 \\
           0&0&0&0 \\
           0&0&0&0 \\
           0& \tfrac{\overline{\mu}}{\theta' p'_\rho}(\partial_1 \mathrm{u}'_2+\partial_2 \mathrm{u}'_1)&2\tfrac{\overline{\mu}}{\theta' p'_\rho}\partial_2 \mathrm{u}'_2+ \tfrac{\overline{\lambda}}{\theta' p'_\rho}\mathrm{div}\mathrm{u}'& 0
       \end{pmatrix}\,.
\end{equation}

The corresponding initial boundary value problems of the linearized Euler equation are
\begin{equation}\label{linearized Euler}
\begin{aligned}
A_0(V')\pa_t V^0+\sum_{j=1}^2A_j(V')\pa_j V^0&=0 \quad \mbox{in}\quad\! \Omega \times [0,T]\,,\\
M^0 V^0& = 0 \quad \mbox{on}\quad\! \Gamma \times [0,T]\,, \\
V^0(x,0)& = V_0 (x)\quad \mbox{for}\quad\! x\in \Omega\,,
\end{aligned}
\end{equation}
where $M^0= (0, 1, 0, 0)$.

Before we state our results, we introduce some function spaces and the notion of compatibility condition. Let $m \geq 0$ be an integer and $U \subset \mathbb{R}^d, d \geq 1$ be a domain. Then $H^m(U)$ denotes the usual Sobolev space of order $m$ equipped with the norm $\|\cdot \|_{m,U}$ and the inner product $(\cdot, \cdot)_{m,U}$. For $0 < \alpha < 1$, $C^\alpha(U)$ denotes the H\"{o}lder space on $\overline{U}$ with the exponent $\alpha$, endowed with the norm $\|\cdot\|_{C^\alpha(U)}$. The space $C^m(I\,; X)$ denotes the set of functions $u(t)$, $t \in I$, the $m\mbox{-}$times continuously differentiable functions on the interval $I$ with value taken in the Banach space $X$.

To study the initial boundary value problem \eqref{linearized NSF}, we need the following compatibility condition: Define inductively the n-Cauchy data of \eqref{linearized NSF} by
\begin{equation}\label{inductive compatbility}
\begin{split}
\dot{\pa}^0_t V^\eps(x,0)=&V^\mathrm{in}(x),\\
\dot{\pa}^n_t V^\eps(x,0)=&\sum^{n-1}_{s=0}C^s_{n-1}\{-\sum^2_{j=1}\pa^s_t(A^{-1}_0(V')A_j(V'))(x,0)
\dot{\pa}^{n-1-s}_t V^\eps(x,0)\\&+\pa^s_t A^{-1}_0(V')\dot{\pa}^{n-1-s}_t B(\eps^2, c_1\eps^2, c_2\eps^2)V^\eps(x,0)\}\quad \mbox{for}\quad\! n=1,2,\cdots
\end{split}
\end{equation}
The initial data $V^\mathrm{in}(x)$ is said to satisfy the compatibility condition of
order $m$ for the initial boundary value problem \eqref{linearized NSF} for any $\varepsilon>0$ if $M^+ \dot{\partial}^n_t V^\eps(x,0)=0$, i.e.
\begin{equation}\label{compatibility C}
\dot{\pa}^n_t v_1(x,0)=\dot{\pa}^n_t v_2(x,0)=\dot{\pa}^n_t \theta(x,0) \quad \mbox{on}\quad\!\Gamma\,,\quad n=0,1,\cdots,m,\quad\mbox{for any}\quad\! \eps >0\,.
\end{equation}
The condition \eqref{compatibility C} implies the corresponding compatibility condition of order $m$ of the initial boundary value problem for linearized
Euler equation \eqref{linearized Euler}:
\begin{equation}
M^0 \dot{\partial}^n_t V^0(x,0) = 0\quad\mbox{on}\quad\! \Gamma\,, \quad\! n=0,1,\cdots, m\,,
\end{equation}
where $\dot{\partial}^n_t V^0(x,0)$ are defined by \eqref{inductive compatbility} in which $\eps$ is taken as zero.

Then we state the main theorem of this paper
\begin{theorem}
Let $m$ be an integer satisfies $m\geq 2(7N+4)$. Suppose that the initial data $V_0\in H^m(\Omega)$ satisfies the
compatibility condition of order $[\frac{m}{2}]-1$ for \eqref{linearized NSF} for any $\varepsilon>0$ and
the compatibility condition of order m-1 for \eqref{linearized Euler}.
Then the solution $V=(\rho,v_1,v_2,\theta)^{\top}$ of problem \eqref{linearized NSF}, the solution
$E=(\rho^E,v^E_1,v^E_2,\theta^E)^{\top}$ of \eqref{linearized Euler} and the correcting term
$K^{\varepsilon}=U^{\varepsilon}-E,$ $W^{\varepsilon}$ is defined in \eqref{expansion}, exist
uniquely in the following spaces:
\begin{equation}
V\in \bigcap^{[m/2]}_{j=0}C^j([0,T];H^{m-2j}(\Omega)),\;\;\;\;E\in\bigcap^{m}_{j=0}C^j([0,T];H^{m-j}(\Omega)),
\end{equation}
\begin{equation}
K^{\varepsilon}\in \bigcap^{[m/2]-1-7N}_{j=0}C^j([0,T];H^{[m/2-1-7N-j]}(\Omega)),
\end{equation}
and for $0<\varepsilon<1,$ there exist constants  $C_1$ and $C_2$ which are independent
 of $\varepsilon$, such that the following estimate hold:
\begin{equation}
\sup_{(x,t)\in \Omega\times [0,T]}|\rho(x,t)-\rho^E(x,t)-K^{\varepsilon}_0(x,t)|\leq C_1 \varepsilon^{N-1},
\end{equation}
\begin{equation}
\sup_{(x,t)\in \Omega\times [0,T]}|v_j(x,t)-v_j^E(x,t)-K^{\varepsilon}_j(x,t)|\leq C_2 \varepsilon^{N-3/4},\;\;j=1.2,
\end{equation}
and
\begin{equation}
\sup_{(x,t)\in \Omega\times [0,T]}|\theta(x,t)-\theta^E(x,t)-K^{\varepsilon}_3(x,t)|\leq C_2 \varepsilon^{N-3/4}.
\end{equation}
\end{theorem}

\section{Construction of an Approximate Solution}
Throughout this section, we denote the solution of \eqref{linearized NSF} by $V$ instead of $V^\eps$ for simplicity.
\subsection{Boundary Characteristic Variables}
In the isentropic case \cite{xy}, the matrix $A_1(V')$ is diagonalized. In this paper, for the non-isentropic case, this diagonalization process is not easy. Instead, we decompose $A_1(V')$ into two parts: one is easy to be diagonalized, the other vanishes on the boundary $\Gamma$. Let $A_{1m}= A_{1m}(x_2,t)=A_1(V'(0,x_2,t))$, then
\begin{equation}\label{11}
\begin{split}
A_1(V')&=A_{1m}(V')+A_{1r}(V')\\& =
\begin{pmatrix}
0&1&0&0\\
1&0&0&\alpha\\
0&0&0&0\\
0&\alpha&0&0
\end{pmatrix}
+
\begin{pmatrix}
\frac{\mathrm{u}'_1}{\rho'}&0&0&0\\
0&\frac{\rho'\mathrm{u}'_1}{p'_{\rho}}&0&\frac{p'_{\theta}}{p'_{\rho}}-\alpha\\
0&0&\frac{\rho'\mathrm{u}'_1}{p'_{\rho}}&0\\
0&\frac{p'_{\theta}}{p'_{\rho}}-\alpha &0&\beta'\mathrm{u}'_1
\end{pmatrix}
\end{split}
\end{equation}
where $\alpha=\alpha(x_2,t)=\frac{p'_{\theta}}{p'_\rho}(0,x_2,t)$ is valued on the boundary $\Gamma= \{(0,x_2): x_2\in \mathbb{R}\}$. Note that the matrix $A_{1r}$ vanishes on the boundary $\Gamma \times [0,T]$. In other words, $A_{1r}$ does not contribute the nonzero eigenvalues and eigenvectors of $A_1$ on the boundary. In this sense, we call $A_{1m}$ the main part of the matrix $A_1(V')$ of the hyperbolic part of the equation \eqref{linearized NSF}.

Simple calculations show that the eigenvalues of $A_{1m}$ are $\lambda_0=\lambda_1=0$, $\lambda_2= \sqrt{\alpha^2 + 1}$, and $\lambda_3= - \sqrt{\alpha^2 + 1}$. Note that for the isentropic case $p_\theta=0$, so the eigenvalues are $0, 1, -1$, which are reduced to the case considered in \cite{xy}. The corresponding right orthonormal eigenvector of $A_{1m}$ are given by
\begin{equation}\label{12}
\begin{split}
e_0&=\left(0,0,1,0\right)^\top,\\
e_1&=\left(\tfrac{\alpha}{\sqrt{\alpha^2+1}},0,0,\tfrac{-1}{\sqrt{\alpha^2+1}}\right)^\top,\\
e_2&=\left(\tfrac{1}{\sqrt{2(\alpha^2+1)}},\tfrac{1}{\sqrt{2}},0,\tfrac{\alpha}{\sqrt{2(\alpha^2+1)}}\right)^\top,\\
e_3&=\left(\tfrac{1}{\sqrt{2(\alpha^2+1)}},-\tfrac{1}{\sqrt{2}},0,\tfrac{\alpha}{\sqrt{2(\alpha^2+1)}}\right)^\top.
\end{split}
\end{equation}
Let $Q= (e_0, e_1, e_2, e_3)^\top$, then $Q$ is an orthogonal matrix, $Q^{-1}= Q^\top$. We define the boundary characteristic variables by
\begin{equation}\nonumber
U=\begin{pmatrix}\mathrm{u}_0\\ \mathrm{u}_1\\ \mathrm{u}_2\\ \mathrm{u}_3\end{pmatrix}=QV=
\begin{pmatrix}
v_2\\\tfrac{\rho-\theta}{\sqrt{\alpha^2 + 1}}\\ \tfrac{\rho -\alpha \theta}{\sqrt{2(\alpha^2+1)}}+ \tfrac{v_1}{2}\\ \tfrac{\rho -\alpha \theta}{\sqrt{2(\alpha^2+1)}}- \tfrac{v_1}{2}
\end{pmatrix}\,,
\end{equation}
in terms of which the linearized Navier-Stokes-Fourier equations \eqref{linearized NSF} can be transformed into
\begin{equation}\label{boundary variable U equa}
\mathcal{A}_0 \pa_t U - \mathcal{L}^\eps U = 0\,,
\end{equation}
where
\begin{equation}\label{L-eps}
\begin{aligned}
 \mathcal{L}^\eps=& \sum^2_{j=1}\left\{-\mathcal{A}_j(x,t) + \eps^2 \mathcal{P}_j(x,t) + \eps^2 \mathcal{I}_j(x,t)\right\}\partial_j\\
 &  + \left( -\mathcal{W}(x,t) + \eps^2 \mathcal{Q}^1(x,t) + \eps^2\mathcal{Q}^2(x,t)\right) + \eps^2 \mathcal{G}(x,t)\Delta + \eps^2 \overline{\xi} \sum^2_{i,j=1}\mathcal{G}^{ij}\partial_{ij}\,.
\end{aligned}
\end{equation}
Here
\begin{equation*}
\begin{split}
&\mathcal{A}_0(x,t)=QA_0Q^{-1}=\\&\begin{pmatrix}
 \tfrac{\rho'}{p'_{\rho}}&0&0&0\\
 0&\tfrac{1}{\alpha^2+1}(t\frac{\alpha^2}{\rho'}+\beta')&
 \tfrac{1}{\alpha^2+1}(\tfrac{\alpha}{\sqrt{2}\rho'}-\tfrac{\alpha\beta'}{\sqrt{2}})
 &\tfrac{1}{\alpha^2+1}(\tfrac{\alpha}{\sqrt{2}\rho'}-\tfrac{\alpha\beta'}{\sqrt{2}})\\
 0&\tfrac{1}{\alpha^2+1}(\tfrac{\alpha}{\sqrt{2}\rho'}-\tfrac{\alpha\beta'}{\sqrt{2}})&\tfrac{1}{\alpha^2+1}
 (\tfrac{1}{2\rho'}+\tfrac{\alpha^2\beta'}{2})+\tfrac{\rho'}{2p'_{\rho}}
 &\tfrac{1}{\alpha^2+1}
 (\tfrac{1}{2\rho'}+\tfrac{\alpha^2\beta'}{2})-\tfrac{\rho'}{2p'_{\rho}}\\
0&\tfrac{1}{\alpha^2+1}(\tfrac{\alpha}{\sqrt{2}\rho'}-\tfrac{\alpha\beta'}{\sqrt{2}})&\tfrac{1}{\alpha^2+1}
 (\tfrac{1}{2\rho'}+\tfrac{\alpha^2\beta'}{2})-\tfrac{\rho'}{2p'_{\rho}}&\tfrac{1}{\alpha^2+1}
 (\tfrac{1}{2\rho'}+\tfrac{\alpha^2\beta'}{2})+\tfrac{\rho'}{2p'_{\rho}}
 \end{pmatrix}\\
&= \begin{pmatrix}
 \tfrac{\rho'}{p'_{\rho}}&0&0&0\\
 0&\eta_0&\eta_1&\eta_1\\
 0&\eta_1&\eta_2&\eta_3\\
 0&\eta_1&\eta_3&\eta_2
 \end{pmatrix}\,.
\end{split}
\end{equation*}
and $\mathcal{A}_1(x,t)= \mathcal{A}_{1m} + \mathcal{A}_{1r}$,
\begin{equation}\label{A-1m}
\mathcal{A}_{1m}=QA_{1m}Q^{-1}=
\begin{pmatrix}
 0&0&0&0\\
 0&0&0&0\\
 0&0&\sqrt{\alpha^2+1}&0\\
0&0&0&-\sqrt{\alpha^2+1}
 \end{pmatrix},
\end{equation}
$$
\mathcal{A}_{1r}=QA_{1r}Q^{-1}\,,\quad \mathcal{A}_2=QA_2Q^{-1}\,.
$$
\begin{equation*}
\mathcal{P}_j(x,t)=2Q
\begin{pmatrix}
0&0&0&0\\
0&\frac{\overline{\mu}}{p'_{\rho}}&0&0\\
0&0&\frac{\overline{\mu}}{p'_{\rho}}&0\\
0&0&0&\frac{\overline{\kappa}}{\theta'p'_{\rho}}
\end{pmatrix}
\pa_j Q^{-1}\,,\quad \mathcal{I}_j(x,t) = Q I_j Q^{-1}\,,
\end{equation*}
\begin{equation*}
\mathcal{W}(x,t)=QA_0\pa_tQ^{-1}+ \sum^2_{j=1}QA_j\pa_j Q^{-1}\,,
\end{equation*}
\begin{equation*}
\mathcal{Q}^1(x,t)= Q
\begin{pmatrix}
0&0&0&0\\
0&\frac{1}{p'_{\rho}}&0&0\\
0&0&\frac{1}{p'_{\rho}}&0\\
0&0&0&\frac{\kappa_0(\theta')}{\theta'p'_{\rho}}
\end{pmatrix}
\Delta Q^{-1}\,, \quad \mathcal{Q}^2(x,t)= \sum^2_{j=1}Q I_j \partial_j Q^{-1}\,,
\end{equation*}
\begin{equation*}
\begin{split}
\mathcal{G}(x,t)&=Q\begin{pmatrix}
0&0&0&0\\
0&\frac{1}{p'_{\rho}}&0&0\\
0&0&\frac{1}{p'_{\rho}}&0\\
0&0&0&\frac{\kappa_0(\theta')}{\theta'p'_{\rho}}
\end{pmatrix}Q^{-1}\\
&=
\begin{pmatrix}
\frac{1}{p'_{\rho}}&0&0&0\\
0&\frac{\kappa_0(\theta')}{\theta'p'_{\rho}(\alpha^2+1)}&\frac{-\alpha \kappa_0(\theta')}{\sqrt{2}(\alpha^2+1)\theta'p'_{\rho}}
&\frac{-\alpha\kappa_0(\theta')}{\sqrt{2}(\alpha^2+1)\theta'p'_{\rho}}\\
0&\frac{-\alpha \kappa_0(\theta')}{\sqrt{2}(\alpha^2+1)\theta'p'_{\rho}}&\frac{1}{2p'_{\rho}}+\frac{\alpha^2 \kappa_0(\theta')}
{2(\alpha^2+1)\theta'p'_{\rho}}&-\frac{1}{2p'_{\rho}}+\frac{\alpha^2 \kappa_0(\theta')}
{2(\alpha^2+1)\theta'p'_{\rho}}\\
0&\frac{-\alpha \kappa_0(\theta')}{\sqrt{2}(\alpha^2+1)\theta'p'_{\rho}}&-\frac{1}{2p'_{\rho}}+\frac{\alpha^2 \kappa_0(\theta')}
{2(\alpha^2+1)\theta'p'_{\rho}}&\frac{1}{2p'_{\rho}}+\frac{\alpha^2 \kappa_0(\theta')}
{2(\alpha^2+1)\theta'p'_{\rho}}
\end{pmatrix}\,,
\end{split}
\end{equation*}
$$
\mathcal{G}^{11}=
\begin{pmatrix}
0&0&0&0\\
0&0&0&0\\
0&0&\frac{1}{2}&-\frac{1}{2}\\
0&0&-\frac{1}{2}&\frac{1}{2}
\end{pmatrix},
\;\;\;
\mathcal{G}^{22}=
\begin{pmatrix}
1&0&0&0\\
0&0&0&0\\
0&0&0&0\\
0&0&0&0
\end{pmatrix},
$$
$$
\mathcal{G}^{12}=\mathcal{G}^{21}=
\begin{pmatrix}
0&0&\frac{\sqrt{2}}{4}&-\frac{\sqrt{2}}{4}\\
0&0&0&0\\
\frac{\sqrt{2}}{4}&0&0&0\\
-\frac{\sqrt{2}}{4}&0&0&0
\end{pmatrix}\,.
$$
We can also rewrite $\mathcal{L}^\eps$ as
$$
\mathcal{L}^\eps = \mathcal{L}^0 + \eps^2 \Lambda\,,
$$
where $\mathcal{L}^0= -\mathcal{A}_j \partial_j - \mathcal{W}$, and
$
\Lambda =\left( \mathcal{Q}^1 + \mathcal{Q}^2\right) +\sum^2_{j=1}\left( \mathcal{P}_j + \mathcal{I}_j\right)\partial_j + \left(\mathcal{G}\Delta +\sum^2_{j=1} \mathcal{G}^{ij}\partial_{ij}\right)\,.
$

The boundary condition of \eqref{linearized NSF} can be rewritten in terms of the boundary characteristic variables as:
\begin{equation}\label{boundary variable BC}
\mathcal{M}^+ U = 0\quad\mbox{on}\quad\! \Gamma \times [0,T]\quad \mbox{where}
\quad\! \mathcal{M}^+=M^+Q^{-1}= \begin{pmatrix}
                           0&0&\tfrac{1}{\sqrt{2}}&\tfrac{-1}{\sqrt{2}}\\
                           1&0&0&0\\
                           0&\tfrac{-1}{\sqrt{\alpha^2+1}}&\tfrac{\alpha}{\sqrt{2(\alpha^2+1)}}&\tfrac{\alpha}{\sqrt{2(\alpha^2+1)}}\\
                       \end{pmatrix}\,,
\end{equation}
i.e.
\begin{equation}\nonumber
\begin{split}
\mathrm{u}_2-\mathrm{u}_3=0\,,\quad \mathrm{u}_0=0\,,\quad -\mathrm{u}_1+\tfrac{\alpha}{\sqrt{2}} \mathrm{u}_2+\tfrac{\alpha}{\sqrt{2}} \mathrm{u}_3=0\,,\quad\text{on}\quad\!\Gamma\times[0,T]\,,
\end{split}
\end{equation}
which can be simplified as
\begin{equation}\label{boundary variable BC-1}
\mathrm{u}_2=\mathrm{u}_3\,,\quad \mathrm{u}_0=0\,,\quad \mathrm{u}_1 = \sqrt{2}\alpha \mathrm{u}_2 \,,\quad\text{on}\quad\!\Gamma\times[0,T]\,.
\end{equation}
The initial conditions of \eqref{linearized NSF} can be rewritten as:
\begin{equation}\label{boundary variable IC}
U(x,0)=QV_0= U_0(x)\,,\quad \text{for}\quad\! x\in\Omega.
\end{equation}
Furthermore, the boundary condition for the linearized Euler equations \eqref{linearized Euler} becomes
\begin{equation}\label{boundary variable BC-Euler}
\mathcal{M}^0 U^0= M^0Q^{-1} U^0 = 0\quad \mbox{on}\quad\! \Gamma \times [0,T]\,,
\end{equation}
i.e.
\begin{equation}\label{boundary variable BC-Euler-1}
\begin{split}
\mathrm{u}_2-\mathrm{u}_3=0\,,\quad\text{on}\quad\!\Gamma\times[0,T]\,.
\end{split}
\end{equation}
It is clear that the initial data $U_0$ satisfies the compatibility condition of order $l \geq 0$ for \eqref{boundary variable U equa}, \eqref{boundary variable BC} and \eqref{boundary variable IC} for any $\eps > 0$ if and only if $V_0$ satisfies the compatibility condition of order $l \geq 0$ for any $\eps >0$. The same statement is also true for the linearized Euler case.

\subsection{Formal Inner and Boundary Expansions}
We construct the approximate solution $U^{\varepsilon}$ of equation \eqref{boundary variable U equa} with boundary and initial conditions \eqref{boundary variable BC}-\eqref{boundary variable IC} as
\begin{equation}\label{expansion}
\begin{aligned}
W^\eps(x,t)& = E^{\varepsilon}(x,t)+B^{\varepsilon}(x,t)\\
& =\sum^N_{i=0}\varepsilon^i E^i(x_1,x_2,t)+\sum^N_{i=0}
\varepsilon^i B^i\left(\frac{x_1}{\varepsilon},x_2,t\right)\,.
\end{aligned}
\end{equation}
Formally, for the inner term $E^\eps$,
\begin{equation}\label{expansion inner}
\begin{aligned}
(\mathcal{A}_0\partial_t-\mathcal{L}^\eps)E^\eps
= & (\mathcal{A}_0\partial_t-\mathcal{L}^0)E^0 + \eps (\mathcal{A}_0\partial_t-\mathcal{L}^0)E^1\\
+&\sum^N_{i=2}\eps^i  \left\{(\mathcal{A}_0\partial_t-\mathcal{L}^0) E^i - \Lambda E^{i-2}\right\} - \eps^{N+1}\Lambda E^{N-1} - \eps^{N+2}\Lambda E^N\,.
\end{aligned}
\end{equation}
For the boundary term $B(\frac{x_1}{\eps}, x_2, t)$,
\begin{equation}
\begin{aligned}
 (\mathcal{A}_0\partial_t-\mathcal{L}^\eps)B =& \frac{1}{\eps}\mathcal{A}_1 \partial_{z_1}B + \left\{ \mathcal{A}_0\partial_t + \mathcal{A}_2 \partial_{z_2}-(\mathcal{G}+ \mathcal{G}^{11})\partial_{z_1z_1} + \mathcal{W}\right\}B\\
 - & \eps \left\{ 2\mathcal{G}^{12}\partial_{z_1z_2}+ (\mathcal{P}_1 + \mathcal{I}_1)\partial_{z_1}\right\}B\\
 - & \eps^2 \left\{(\mathcal{G}+ \mathcal{G}^{22})\partial_{z_2z_2}+ (\mathcal{P}_2 + \mathcal{I}_2)\partial_{z_2} + (\mathcal{Q}^1 + \mathcal{Q}^2) \right\}B\,.
\end{aligned}
\end{equation}
Note that in the above expansion of the differential operator $\mathcal{A}_0\partial_t-\mathcal{L}^\eps$, all the coefficients are functions of $(x_1, x_2,t)=(\eps z_1, z_2, t)$. In the construction of the boundary layer part of the approximate solution $B^\eps$, we only concern the values on the boundary $\Gamma \times [0,T]$. So we expand all the coefficients around $(0,z_2,t)$ as follows: We use the notation that, for a smooth matrix-valued function $\mathcal{K}(x_1,x_2,t)=\mathcal{K}(\varepsilon z_1,z_2,t)$, the Taylor expansion around $(0,z_2, t)$ is
 \begin{equation*}
\begin{split}
& \mathcal{K}(\varepsilon z_1,z_2,t)\\=&\sum_{i=0}^N(\varepsilon z_1)^i \frac{\pa^i_{x_1} \mathcal{K}(0,z_2,t)}{i!}+
\frac{(\varepsilon z_1)^{N+1}}{N!}\int_0^1\pa^{N+1}_{x_1}\mathcal{K}(\varepsilon z_1 \xi,z_2,t))(1-\xi)^Nd\xi\\
= &\sum_{i=0}^N(\varepsilon z_1)^i \mathcal{K}^i(z_2,t)+(\varepsilon z_1)^{N+1}\mathcal{K}^{N+1}_R\,.
\end{split}
\end{equation*}
In particular, $
\mathcal{A}_1(\eps z_1, z_2, t)= \mathcal{A}_{1m}(z_2,t)+ \frac{1}{i!}\sum_{i=0}^N(\eps z_1)^i \partial^i_{x_1}\mathcal{A}_{1r}(0,z_2,t)+(\varepsilon z_1)^{N+1}(\mathcal{A}_{1r}^{N+1})_R\,.$ Thus
\begin{equation}\label{boundary operator expansion}
(\mathcal{A}_0\partial_t-\mathcal{L}^\eps)B = \left\{\frac{1}{\eps}\mathcal{L}^\b_{-1} + \mathcal{L}^\b_0 + \eps \mathcal{L}^\b_1 + \eps^2\mathcal{L}^\b_2\right\} B + \text{higher order terms}\,,
\end{equation}
where
\begin{equation}\nonumber
\begin{aligned}
 &\mathcal{L}^\b_{-1} =\mathcal{A}_1(z_2,t)\partial_{z_1}\,,\\
 &\mathcal{L}^\b_0  =  \mathcal{A}_0(z_2,t)\partial_t + \mathcal{A}_2(z_2,t)\partial_{z_2} - (\mathcal{G}+ \mathcal{G}^{11})(z_2,t)\partial^2_{z_1z_1}
                 + \mathcal{W}(z_2,t) + z_1\partial_{x_1}\mathcal{A}_{1r}(z_2,t)\partial_{z_1}\,,\\
 &\mathcal{L}^\b_1 =  2\mathcal{G}^{12}(z_2,t)\partial^2_{z_1z_2}+ (\mathcal{P}_1+ \mathcal{I}_1+ \tfrac{1}{2}z^2_1\partial^2_{x_1}\mathcal{A}_{1r})(z_2,t)\partial_{z_1} + z_1\partial_{x_1}\mathcal{A}_0(z_2,t)\partial_t\,,\\
  &\mathcal{L}^\b_2 =  (\mathcal{G}+ \mathcal{G}^{22})(z_2,t)\partial^2_{z_2z_2}+ (\mathcal{P}_2+ \mathcal{I}_2)(z_2,t)\partial_{z_2}+ \tfrac{1}{6}z^3_1\partial^3_{x_1}\mathcal{A}_{1r}(z_2,t)\partial_{z_1}
   + \tfrac{1}{2}z^2_1\partial^2_{x_1}\mathcal{A}_{1r}(z_2,t)\partial_t\,.
\end{aligned}
\end{equation}
Here we use the notation $\mathcal{K}(z_2,t)=\mathcal{K}(0,z_2,t)$ for a function $\mathcal{K}$. In \eqref{boundary operator expansion}, the precise forms of the  ``higher order term" are lengthy and not important for the later analysis, so we omit writing out the details. In fact, only $\mathcal{L}^\b_{-1}$ and $\mathcal{L}^\b_0$ play important roles in the later boundary layer analysis. Note that
\begin{equation}\label{A2-boundary}
\mathcal{A}_2(z_2,t) = \sqrt{\tfrac{\alpha^2+1}{2}}\begin{pmatrix}
                             0&0&1&1\\
                             0&0&0&0\\
                             1&0&0&0\\
                             1&0&0&0\\
                         \end{pmatrix}
\end{equation}

Thus,
\begin{equation}\label{expansion boundary}
\begin{aligned}
 &(\mathcal{A}_0\partial_t-\mathcal{L}^\eps)B^\eps\\
  =& \frac{1}{\eps}\mathcal{L}^\b_{-1}B^0 + \{\mathcal{L}^\b_{-1}B^1+ \mathcal{L}^\b_0 B^0 \} + \eps \{ \mathcal{L}^\b_{-1}B^2 + \mathcal{L}^\b_0B^1 + \mathcal{L}^\b_1 B^0\}\\
 + & \sum^{N-1}_{i=2} \eps^i \{\mathcal{L}^\b_{-1}B^{i+1} + \mathcal{L}^\b_0B^i + \mathcal{L}^\b_1 B^{i-1} + \mathcal{L}^\b_2 B^{i-2} \}\\
 + & \eps^N \{\mathcal{L}^\b_0B^N + \mathcal{L}^\b_1 B^{N-1} + \mathcal{L}^\b_2 B^{N-2} \} + \eps^{N+1}\{\mathcal{L}^\b_1 B^N + \mathcal{L}^\b_2 B^{N-1}\} + \eps^{N+2} \mathcal{L}^\b_2 B^N\\
 + & \text{higher order terms}\,.
\end{aligned}
\end{equation}

\subsection{Construction of Inner and Boundary Expansions}
We now construct the inner and boundary expansions at each order in details. We plug $W^\eps = E^\eps + B^\eps$ into the equation \eqref{boundary variable U equa} and compare the coefficients of the same order in both inner and boundary terms. It should be noted that the same order of inner and boundary functions will be constructed simultaneously due to their coupling at the boundary.

{\bf \underline{Order $O(\tfrac{1}{\eps})$ of boundary expansion}:}

We start from the first-order term $B^0$ in the boundary expansion by setting the order $O(\frac{1}{\eps})$ in the boundary part zero gives $\mathcal{L}^\b_{-1}B^0=0$, i.e
\begin{equation}\label{B-0-equation-1}
   \mathcal{A}_1(0,z_2,t) \partial_{z_1}B^0(z_1, z_2, t)= 0\,.
\end{equation}
Note that $\mathcal{A}_1(0,z_2,t)= \mathcal{A}_{1m}(0,z_2,t)$ since $\mathcal{A}_{1r}$ vanishes on the boundary $\Gamma \times [0,T]$, i.e. $\mathcal{A}_{1r}(0,z_2,t)=0$. Noting \eqref{A-1m}, the equation \eqref{B-0-equation-1} is equivalent to
\begin{equation}\label{B-0-equation}
 \begin{pmatrix}
    \sqrt{\alpha^2+1}&0\\
    0& -\sqrt{\alpha^2+1}
 \end{pmatrix} \begin{pmatrix}
                   \partial_{z_1}B^0_2\\
                   \partial_{z_1}B^0_3
               \end{pmatrix} =\begin{pmatrix}
                   0\\
                  0
               \end{pmatrix}\quad \mbox{in}\quad\! \Omega \times [0,T]\,.
\end{equation}
Note that we impose the decay condition at infinity that
\begin{equation}\label{B-infity}
 B^0_j(z_1, z_2, t) \rightarrow 0\quad \mbox{as}\quad\! z_1\rightarrow \infty\,,\quad\! (z_2,t)\in \mathbb{R}\times [0,T]\,.
\end{equation}
The only solution of \eqref{B-0-equation} and \eqref{B-infity} is given as
\begin{equation}\label{B0-II-solution}
 B^0_2(z_1,z_2,t)= B^0_3(z_1, z_2, t) \equiv 0 \quad\mbox{for}\quad\! z_1\geq 0\,,\quad\!(z_2,t)\in \mathbb{R}\times [0,T]\,.
\end{equation}

{\bf \underline{ Order $O(1)$ of inner expansion}:}

We determine the leading order term $E^0$ in the inner expansion by setting the $O(1)\mbox{-}$order term in \eqref{expansion inner} to zero and then equipping the resulting equations with the same initial and boundary conditions as in
\eqref{boundary variable BC} and \eqref{boundary variable BC-Euler}. So we deduce the following initial boundary value problem for $E^0$:
\begin{equation}\label{equation E0}
\begin{aligned}
\mathcal{A}_0 \pa_t E^0 - \mathcal{L} E^0=0\,,\quad &\text{in}\quad\!\Omega\times[0,T]\,,\\
\mathcal{M}^0 (E^0 + B^0) =0\,, \quad &\text{on}\quad\!\Gamma\times[0,T]\,,\\
E^0(x,0)=U_0(x)\,,\quad &\text{for}\quad x\in\Omega\,.
\end{aligned}
\end{equation}
Note that from the definition of $\mathcal{M}^0$, see \eqref{boundary variable BC-Euler-1}, only the third and fourth components of $B^0$ and $E^0$ are involved, and $B^0_2$ and $B^0_3$ are solved in \eqref{B0-II-solution}. Thus the boundary condition for $E^0$ in \eqref{equation E0} is $\mathcal{M}^0 E^0=0$, more specifically, $E^0_2-E^0_3 =0$.

It is easy to see that $Q^{-1}E^0$ is a solution of the initial boundary value problem of the linearized Euler equations with the same boundary and initial conditions as in \eqref{linearized Euler}. Then it follows by Proposition \ref{existence linarized Euler} that there exists a unique $E^0$ of the problem \eqref{equation E0}, such that
\begin{equation}\label{21}
E^0\in \bigcap_{j=0}^m C^j([0,T];H^{m-j}(\Omega))\,.
\end{equation}

{\bf \underline{Order $O(1)$ of boundary expansion}:}

By setting the term of $O(1)$ in \eqref{expansion boundary} equal to zero, we have
\begin{equation*}
\mathcal{L}^\b_0 B^0 + \mathcal{L}^\b_{-1}B^1 = 0 \,,
\end{equation*}
which gives
\begin{equation}\label{B0-B1}
\begin{aligned}
&\begin{pmatrix}
 \tfrac{\rho'}{p'_{\rho}}&0&0&0\\
 0&\eta_0&\eta_1&\eta_1\\
 0&\eta_1&\eta_2&\eta_3\\
 0&\eta_1&\eta_3&\eta_2
 \end{pmatrix}\partial_t  \begin{pmatrix}
                                                B^0_0\\
                                                B^0_1\\
                                               0\\
                                                0
                                            \end{pmatrix} + \sqrt{\tfrac{\alpha^2+1}{2}}\begin{pmatrix}
                             0&0&1&1\\
                             0&0&0&0\\
                             1&0&0&0\\
                             1&0&0&0\\
                         \end{pmatrix} \partial_{z_2}\begin{pmatrix}
                                                B^0_0\\
                                                B^0_1\\
                                               0\\
                                                0
                                            \end{pmatrix}\\
-&\begin{pmatrix}
 \tfrac{1}{p'_{\rho}}&0&0&0\\
 0&\tau_0&\tau_1&\tau_1\\
 0&\tau_1&\tau_2&\tau_3\\
 0&\tau_1&\tau_3&\tau_2
 \end{pmatrix}\partial^2_{z_1z_1}\begin{pmatrix}
                                                B^0_0\\
                                                B^0_1\\
                                               0\\
                                                0
                                            \end{pmatrix} + \left\{\mathcal{W}(z_2,t)+ z_1 \partial_{x_1}\mathcal{A}_{1r}(z_2,t)\partial_{z_1}\right\}\begin{pmatrix}
                                                B^0_0\\
                                                B^0_1\\
                                               0\\
                                                0
                                            \end{pmatrix}\\
 = & -\begin{pmatrix}
 0&0&0&0\\
 0&0&0&0\\
 0&0&\sqrt{\alpha^2+1}&0\\
0&0&0&-\sqrt{\alpha^2+1}
 \end{pmatrix} \partial_{z_1} \begin{pmatrix}
                                                B^1_0\\
                                                B^1_1\\
                                                B^1_2\\
                                                B^1_3
                                            \end{pmatrix}\,.
\end{aligned}
\end{equation}

Noticing that the second term in the first line of \eqref{B0-B1} vanishes, i.e. there are no $\partial_{z_2}$ terms in \eqref{B0-B1},  the first two components of \eqref{B0-B1} give the equations of $B^0_0$ and $B^0_1$ which are Prandtl-type linearly coupled equations:
\begin{equation}\label{B0-B1-Prandtl equation-1}
\begin{aligned}
   \tfrac{\rho'}{p'_\rho}(z_2,t)\partial_t B^0_0 - \tfrac{1}{p'_\rho}(z_2,t)\partial^2_{z_1z_1}B^0_0 &+  \left(a^r_{11}z_1\partial_{z_1}B^0_0 + a^r_{12}z_1\partial_{z_1}B^0_1 \right)\\
   &+\left( w_{11}B^0_0 + w_{12}B^0_1\right)=0\,,\\
\end{aligned}
\end{equation}
and
\begin{equation}\label{B0-B1-Prandtl equation-2}
\begin{aligned}
     \eta_0(z_2,t)\partial_t B^0_1 - \tau_0(z_2,t)\partial^2_{z_1z_1}B^0_1 & +  \left(a^r_{21}z_1\partial_{z_1}B^0_0 + a^r_{22}z_1\partial_{z_1}B^0_1 \right)\\
     & +\left( w_{21}B^0_0+ w_{22}B^0_1\right)=0\,,
\end{aligned}
\end{equation}
with the boundary conditions
\begin{equation}\label{B0-B1-BC}
B^0_0=-E^0_0,\;\;\;B^0_1=-E^0_1+ \sqrt{2}\alpha E^0_2\quad \text{on}\quad\!\Gamma\times[0,T],
\end{equation}
and the initial conditions
\begin{equation}\label{B0-B1-IC}
B^0_0(z,0)=B^0_1(z,0)=0\quad \text{for}\quad\! z\in\Omega\,.
\end{equation}
We denote \eqref{B0-B1-Prandtl equation-1} and \eqref{B0-B1-Prandtl equation-2} as
\begin{equation}\label{B0-B1 equation}
\mathcal{E}(B^0_0,B^0_1)=0.
\end{equation}
To solve \eqref{B0-B1 equation} we need to verify the compatibility condition. So we set
\begin{equation}\nonumber
\begin{split}
\tilde{B}^0_0&=B^0_0(z_1,z_2,t)+E^0_0(0,z_2,t)e^{-z_1^2}= B^0_0+\tilde{E}^0_0,\\
\tilde{B}^0_1&=B^0_1(z_1,z_2,t)+(E^0_1(0,z_2,t)- \sqrt{2}\alpha E^0_2(0,z_2,t))e^{-z_1^2}\\& = B^0_1+\tilde{E}^0_1\,.
\end{split}
\end{equation}
Then $(\tilde{B}^0_0,\tilde{B}^0_1)$ satisfies
\begin{equation}\nonumber
\mathcal{E}(\tilde{B}^0_0,\tilde{B}^0_1)=\mathcal{E}(\tilde{E}^0_0,\tilde{E}^0_1).
\end{equation}

As $V_0$ satisfies the compatibility condition of order $[\frac{m}{2}]-1,$ for the problem \eqref{linearized NSF}
for any $\varepsilon>0,$ one derive that
\begin{equation}\nonumber
\pa^k_t\tilde{E}^0_i(z,0)=0\,,\quad k=0,1,\cdots ,[\tfrac{m}{2}]-1,
\end{equation}
and one also could check that
\begin{equation}\nonumber
\langle z_1\rangle^l\pa^k_t\pa^{\alpha_1}_{z_1}\pa^{\alpha_2}_{z_2}[\mathcal{E}(\tilde{E}^0_0,\tilde{E})]\in
C^0([0,T];L^2(\Omega))\quad \text{for}\quad\!
k+|\alpha|\leq m-3,
\end{equation}
Then by Proposition \ref{Prandtl Existence} which will be presented in the last section, we obtain a unique solution $(B^0_1, B^0_1)$ to \eqref{B0-B1-Prandtl equation-1}-\eqref{B0-B1-Prandtl equation-2}-\eqref{B0-B1-BC}-\eqref{B0-B1-IC} such that
\begin{equation*}
\langle z_1\rangle^l\pa^k_t\pa^{\alpha_1}_{z_1}\pa^{\alpha_2}_{z_2}\tilde{B}^0_i\in
C^0([0,T];L^2(\Omega))\,,\quad i=0,1\,,
\end{equation*}
for
$k+|\alpha|\leq m-4,\;k+[\frac{\alpha_1+1}{2}]\leq [\frac{m}{2}]-2,$ and
\begin{equation*}
\pa^k_t \tilde{B}^0_i(z,0)=0\,,\quad \text{for}\quad\!k=0\,,1\,,\cdots\,,[\tfrac{m}{2}]-2\,,\quad\! i=0,1.
\end{equation*}
Thus we have
\begin{equation}\label{B0-I decay}
\langle z_1\rangle^l\pa^k_t\pa^{\alpha_1}_{z_1}\pa^{\alpha_2}_{z_2}B^0_i\in
C^0([0,T];L^2(\Omega))\quad i=0,1,
\end{equation}
for
$k+|\alpha|\leq m-4,\;k+[\frac{\alpha_1+1}{2}]\leq [\frac{m}{2}]-2,$ and
\begin{equation}\label{compatibility B1-I}
\pa^k_t
B^0_i(z,0)=0\,, \quad \text{for}\quad\! k=0\,,1\,,\cdots\,,[\tfrac{m}{2}]-2\,,\quad\! i=0,1.
\end{equation}

The third and the fourth equations of \eqref{B0-B1} could be written as the following ODEs for $z_1$ ($z_2$ and $t$ are parameters):
\begin{equation}\label{38}
\begin{pmatrix}
\sqrt{\alpha^2+1}&0\\
0&-\sqrt{\alpha^2+1}
\end{pmatrix}
\pa_{z_1}\begin{pmatrix}
B^1_2\\
B^1_3
\end{pmatrix}(z_1,z_2,t)=
\begin{pmatrix}
H^1_2(B^0)\\
H^1_3(B^0)
\end{pmatrix}(z_1,z_2,t)\,,
\end{equation}
where $H^1_j\;(j=2,3)$ are linear functions of the known functions $B^0_0$ and $B^0_1$. By the condition \eqref{B-infity}, we have
\begin{equation}\label{B1-II-solution}
B^1_j(z_1,z_2,t)=\int^{\infty}_{z_1}(-1)^{j}\sqrt{\alpha^2+1}H^1_j(B^0)(\xi,z_2,t)\,\mathrm{d}\xi\,, \quad j=2,3.
\end{equation}
It follows from \eqref{B0-I decay} that
\begin{equation}\label{B0-II decay}
\langle z_1\rangle^l\pa^k_t\pa^{\alpha_1}_{z_1}\pa^{\alpha_2}_{z_2}B^1_i\in
C^0([0,T];L^2(\Omega))\quad i=2,3,
\end{equation}
for
$k+|\alpha|\leq m-5,\;k+[\frac{\alpha_1+1}{2}]\leq [\frac{m}{2}]-2,$ and
\begin{equation}\label{compatibility B1-II}
\pa^k_t
B^1_i(z,0)=0\,, \quad \text{for}\quad\! k=0\,,1\,, \cdots\,,[\tfrac{m}{2}]-3\,, \quad i=2,3.
\end{equation}

{\bf \underline{ Order $O(\eps)$ of inner expansion}:}

By setting the order $O(\eps)$ in the inner expansion \eqref{expansion inner}, we are led to the follow initial boundary value problem of $E^1$:
\begin{equation}\label{equation E1}
\begin{split}
(\mathcal{A}_0\partial_t-\mathcal{L}^0)E^1=0\,,\quad &\text{in}\quad\!\Omega\times[0,T],\\
\mathcal{M}^0 (E^1+ B^1)= 0\,, \quad & \text{on}\quad\!\Gamma\times[0,T]\\
E^1(z,0)=0\,,\quad &\text{for}\quad \! z\in\Omega\,.
\end{split}
\end{equation}
Again, note that $\mathcal{M}^0 B^1$ does not contain $B^1_0$ and $B^1_1$, only contains $B^1_2$ and $B^1_3$ which are solved in the last step, see \eqref{B1-II-solution}. So the boundary condition $\mathcal{M}^0 E^1 = -\mathcal{M}^0 B^1$ is known, i.e. $E^1_2-E^1_3=-B^1_2+B^1_3$ on $\Gamma \times [0,T]$. To obtain the existence of the unique solution $E^1$ to \eqref{equation E1}, we need verify the compatibility conditions.  We set
\begin{equation}\nonumber
\tilde{E}^1=E^1-(0,0,B^1_2(0,z_2,t),B^1_3(0,z_2,t))^{\top}e^{-z_1^2}=  E^1-\bar{B}^1\,.
\end{equation}
It follows from \eqref{equation E1} and \eqref{compatibility B1-II}, $\tilde{E}^1$ satisfies the equation
\begin{equation}\label{equation E1 tilde}
\begin{split}
(\mathcal{A}_0\partial_t-\mathcal{L}^0)\tilde{E}^1&=- (\mathcal{A}_0\partial_t-\mathcal{L}^0)\bar{B}^1\,, \quad \text{in}\quad\! \Omega\times[0,T],\\
\tilde{E}^1_2-\tilde{E}^1_3&=0\,, \quad \text{on}\quad\!\Gamma\times[0,T]\\
\tilde{E}^1(x,0)&=0\,, \quad \text{for}\quad\! x\in\Omega.
\end{split}
\end{equation}
It follows from \eqref{B0-II decay} that
$$(\mathcal{A}_0\partial_t-\mathcal{L}^0)\bar{B}^1 \in H^{[\frac{m}{2}]-4}(\Omega\times[0,T]).$$
and $\pa^k_t\mathcal{F}(\bar{B}^1)(x,0)=0\,, k=0\,,1\,,\cdots\,, [\tfrac{m}{2}]-4$, then by
Proposition \ref{existence linarized Euler}, there exists a unique solution to \eqref{equation E1 tilde}
\begin{equation}\nonumber
\tilde{E}^1\in \bigcap_{j=0}^{[\frac{m}{2}]-4}C^j([0,T];H^{[\frac{m}{2}]-4-j}),
\end{equation}
this gives a unique solution $E^1$ to \eqref{equation E1} such that
\begin{equation}\nonumber
E^1\in \bigcap_{j=0}^{[\frac{m}{2}]-4}C^j([0,T];H^{[\frac{m}{2}]-4-j}).
\end{equation}

{\bf \underline{Order $O(\eps)$ of boundary expansion}:}

Similar as before, we next construct $(B^1_0,B^1_1)$ and $(B^2_2,B^2_3)$ by setting the $O(\varepsilon)$-order term
in \eqref{expansion boundary} equal to zero which gives
\begin{equation}\label{B1-B2 equation}
\begin{aligned}
\mathcal{L}^\b_{-1} B^2 + \mathcal{L}^\b_{0}B^1 &= - \mathcal{L}^\b_1 B^0\,,\quad \text{in}\quad\! \Omega\times [0,T]\\
\mathcal{M}^+ (B^1+ E^1)&= 0\,,\quad \text{on}\quad\! \Gamma\times [0,T]\\
B^1(z,0) & =0\,,\quad\text{for}\quad\! z\in \Omega\,.
\end{aligned}
\end{equation}
The first two components of \eqref{B1-B2 equation} are a linear system of Prandtl-type equations
\begin{equation}\label{B1-B2-Prandtl equation}
\mathcal{E}(B^1_0,B^1_1)= F_1(B^0,(B^1)_{\mathrm{II}})\,,
\end{equation}
where the inhomogeneous term $F_1(B^0,(B^1)_{\mathrm{II}})$ is a two components vector-valued function:
\begin{equation}\label{F1}
\begin{aligned}
  F_1(B^0,(B^1)_{\mathrm{II}})= &-\begin{pmatrix}
                                    0\\
                                    \eta_1 \partial_t(B^1_2 + B^1_3)
                                 \end{pmatrix}
                                 - \sqrt{\tfrac{\alpha^2+1}{2}}\begin{pmatrix}
                                         \partial_{z_2}(B^1_2 + B^1_3)\\
                                         0
                                   \end{pmatrix}
                                   + \begin{pmatrix}
                                    0\\
                                    \tau_1 \partial^2_{z_1z_1}(B^1_2 + B^1_3)
                                 \end{pmatrix}\\
  & + (\mathcal{W} + z_1 \partial_{x_1}\mathcal{A}_{1r})\partial_{z_1}(0,0,B^1_2,B^1_3)^\top - (\mathcal{L}^\b_1 B^0)_{\mathrm{I}}\,.
\end{aligned}
\end{equation}
Here we use the notation that for a four components vector $U$, $U_{\mathrm{I}}$ denotes the first two components, and $U_{\mathrm{II}}$ denotes the last two components. Note that $E^1$ is already solved in the last step, so the boundary and initial condition of \eqref{B1-B2-Prandtl equation} are
\begin{equation}\label{B1-B2 Prandtl BC-IC}
\begin{split}
B^1_0=-E^1_0,\;\;B^1_1=-E^1_1+ \sqrt{2}\alpha E^1_2\,,\quad \text{on}\;\Gamma\times [0,T],\\
(B^1_0,B^1_1)^{\top}(z,0)=(0,0)^{\top}\,,\quad z\in\Omega\,.
\end{split}
\end{equation}
Similar as in solving $(B^0_0,B^1_0)$,  set
\begin{equation}\label{50}
\begin{split}
\tilde{B}^1_0&=B^1_0(z_1,z_2,t)+E^1_0(0,z_2,t)e^{-z_1^2} = B^1_0+\tilde{E}^1_0,\\
\tilde{B}^1_1&=B^1_1(z_1,z_2,t)+ (E^1_1(0,z_2,t)- \sqrt{2}\alpha E^1_2(0,z_2,t)) e^{-z_1^2}\\& =  B^1_1+\tilde{E}^1_1,
\end{split}
\end{equation}
Then Proposition \ref{Prandtl Existence} shows that there exists a unique solution $(B^1_0, B^1_1)$ to \eqref{B1-B2-Prandtl equation}-\eqref{B1-B2 Prandtl BC-IC} such that
\begin{equation}\nonumber
\langle z_1\rangle^l\pa^k_t\pa^{\alpha_1}_{z_1}\pa^{\alpha_2}_{z_2}B^1_i\in
C^0([0,T];L^2(\Omega))\quad i=0,1,\quad\! k+|\alpha|\leq[\tfrac{m}{2}]-8,
\end{equation}
and
\begin{equation}\nonumber
\pa^k_tB^1_i(z,0)=0\quad i=0,1,\quad k=0,1,\cdots ,[\tfrac{m}{2}]-8.
\end{equation}

The third and fourth components of \eqref{B1-B2 equation} are ODEs for $(B^2_2, B^2_3)$:
\begin{equation}\label{B2-ODE}
\begin{pmatrix}
\sqrt{\alpha^2+1}&0\\
0&-\sqrt{\alpha^2+1}
\end{pmatrix}
\pa_{z_1}\begin{pmatrix}
B^2_2\\
B^2_3
\end{pmatrix}(z_1,z_2,t)=
\begin{pmatrix}
H^2_2(B^0, B^1)\\
H^2_3(B^0, B^1)
\end{pmatrix}(z_1,z_2,t)\,,
\end{equation}
where $H^2_j\;(j=2,3)$ are linear functions of the known functions $B^0$ and $B^1$. The equation \eqref{B2-ODE} is regarded as an ordinary equation for $B^2_j$ for $j=2,3$ with independent variable $z_1$ and parameters $(z_2,t)\in \mathbb{R}\times [0,T]$. The condition at infinity are imposed as
\begin{equation}\label{B2-infity}
 B^2_j(z_1, z_2, t) \rightarrow 0\quad \mbox{as}\quad\! z_1\rightarrow \infty\,,\quad\! (z_2,t)\in \mathbb{R}\times [0,T]\,.
\end{equation}
Thus, the solution to \eqref{B2-ODE}-\eqref{B2-infity} are uniquely given by
\begin{equation}\label{B2-II-solution}
B^2_j(z_1,z_2,t)=\int^{\infty}_{z_1}(-1)^{j}\sqrt{\alpha^2+1}H^2_j(B^0,B^1)(\xi,z_2,t)\,\mathrm{d}\xi\,, \quad j=2,3
\end{equation}
and they have the following properties
\begin{equation}\label{54}
\langle z_1\rangle^l\pa^k_t\pa^{\alpha_1}_{z_1}\pa^{\alpha_2}_{z_2}B^2_j\in
C^0([0,T];L^2(\Omega))\quad j=2,3,\quad k+|\alpha|\leq[\tfrac{m}{2}]-9.
\end{equation}

{\bf \underline{General cases}: }

For general $i \geq 2$, to solve $E^i$ and $B^i$, it includes 3 steps: The {\bf \underline{first step}} is to solve $(B^i_2, B^i_3)$ by ODEs which comes from the third and fourth components of equation $\mathcal{L}^\b_{-1}B^{i} + \mathcal{L}^\b_0B^{i-1} + \mathcal{L}^\b_1 B^{i-2} + \mathcal{L}^\b_2 B^{i-3}=0$:
\begin{equation}\label{Bi-ODE}
\begin{pmatrix}
\sqrt{\alpha^2+1}&0\\
0&-\sqrt{\alpha^2+1}
\end{pmatrix}
\pa_{z_1}\begin{pmatrix}
B^i_2\\
B^i_3
\end{pmatrix}(z_1,z_2,t)=
\begin{pmatrix}
H^i_2(B^0, B^1, \cdots, B^{i-1})\\
H^i_3(B^0, B^1, \cdots, B^{i-1})
\end{pmatrix}(z_1,z_2,t)\,,
\end{equation}
with the condition at infinity
\begin{equation}\label{Bi-infity}
 B^i_j(z_1, z_2, t) \rightarrow 0\quad \mbox{as}\quad\! z_1\rightarrow \infty\,,\quad\! (z_2,t)\in \mathbb{R}\times [0,T]\,.
\end{equation}
Thus
\begin{equation}\label{Bi-II-solution}
B^i_j=\int^{\infty}_{z_1}(-1)^{j}\sqrt{\alpha^2+1}H^2_j(B^0, B^1, \cdots, B^{i-1})(\xi,z_2,t)d\xi\,,\quad j=2,3
\end{equation}
and they have the following properties
\begin{equation}\label{54}
\langle z_1\rangle^l\pa^k_t\pa^{\alpha_1}_{z_1}\pa^{\alpha_2}_{z_2}B^i_j\in
C^0([0,T];L^2(\Omega))\quad j=2,3,\quad\! k+|\alpha|\leq[\tfrac{m}{2}]+5-7i\,.
\end{equation}

The {\bf \underline{second step}} is to solve $E^i$ by setting the order $O(\eps^i)$ in the inner expansion \eqref{expansion inner} to zero to derive the equation of $E^i$ which is a linearized Euler equation with inhomogeneous term:
\begin{equation}\label{equation Ei}
\begin{aligned}
 (\mathcal{A}_0\partial_t-\mathcal{L}^0) E^i = \Lambda E^{i-2}\,, \quad &\text{in}\quad\! \Omega \times [0,T]\,,\\
 \mathcal{M}^0 E^i =  \mathcal{M}^0 B^i \,, \quad &\text{on}\quad\! \Gamma \times [0,T]\,,\\
E^1(z,0)=0\,,\quad &\text{for}\quad \! z\in\Omega\,.
\end{aligned}
\end{equation}
Since $\mathcal{M}^0 B^i$ does not contain $B^i_0$ and $B^i_1$, only contains $B^i_2$ and $B^i_3$ which are solved in the last step, see \eqref{Bi-II-solution}. So the boundary condition $\mathcal{M}^0 E^i = -\mathcal{M}^0 B^i$ is known, i.e. $E^i_2-E^i_3=-B^i_2+B^i_3$ on $\Gamma \times [0,T]$. Similar as before, after verifying the compatibility conditions, we obtain a unique solution $E^i$ to the initial boundary problem \eqref{equation Ei} by Proposition \ref{existence linarized Euler}:
\begin{equation}\nonumber
E^i\in\bigcap_{j=0}^{[\frac{m}{2}]+3-7i}C^j([0,T];H^{[\frac{m}{2}]+3-7i-j}(\Omega))\,.
\end{equation}

The {\bf \underline{third step}} is to solve $(B^i_0, B^i_1)$ by setting the $O(\varepsilon^i)$-order term
in \eqref{expansion boundary} equal to zero which gives

\begin{equation}\label{Bi-Bi+1 equation}
\begin{aligned}
\mathcal{L}^\b_{-1} B^{i+1} + \mathcal{L}^\b_{0}B^i &= - (\mathcal{L}^\b_1 B^{i-1}+ \mathcal{L}^\b_1 B^{i-2})\,,\quad \text{in}\quad\! \Omega\times [0,T]\\
\mathcal{M}^+ (B^i+ E^i)&= 0\,,\quad \text{on}\quad\! \Gamma\times [0,T]\\
B^i(z,0) & =0\,,\quad\text{for}\quad\! z\in \Omega\,.
\end{aligned}
\end{equation}
The first two components of \eqref{Bi-Bi+1 equation} are a linear system of Prandtl-type equations
\begin{equation}\label{Bi-Bi+1-Prandtl equation}
\mathcal{E}(B^i_0,B^i_1)= F_i(B^0, B^1, \cdots, B^{i-1},(B^i)_{\mathrm{II}})\,,
\end{equation}
where the inhomogeneous term $F_i(B^0, B^1, \cdots, B^{i-1},(B^i)_{\mathrm{II}})$ is a two components vector-valued function:
\begin{equation}\label{Fi}
\begin{aligned}
&F_i(B^0, B^1, \cdots, B^{i-1},(B^i)_{\mathrm{II}})\\
= &-\begin{pmatrix}
                                    0\\
                                    \eta_1 \partial_t(B^i_2 + B^i_3)
                                 \end{pmatrix}
                                 - \sqrt{\tfrac{\alpha^2+1}{2}}\begin{pmatrix}
                                         \partial_{z_2}(B^i_2 + B^i_3)\\
                                         0
                                   \end{pmatrix}
                                   + \begin{pmatrix}
                                    0\\
                                    \tau_1 \partial^2_{z_1z_1}(B^i_2 + B^i_3)
                                 \end{pmatrix}\\
  & + (\mathcal{W} + z_1 \partial_{x_1}\mathcal{A}_{1r})\partial_{z_1}(0,0,B^i_2,B^i_3)^\top - (\mathcal{L}^\b_1 B^{i-1}+ \mathcal{L}^\b_2 B^{i-2})_{\mathrm{I}}\,.
\end{aligned}
\end{equation}
Note that $E^1$ is already solved in the last step, so the boundary and initial condition of \eqref{B1-B2-Prandtl equation} are
\begin{equation}\label{Bi-Bi+1 Prandtl BC-IC}
\begin{split}
B^i_0=-E^1_0\,,\quad B^i_1=-E^i_1+ \sqrt{2}\alpha E^i_2\,,\quad \text{on}\;\Gamma\times [0,T],\\
(B^i_0,B^i_1)^{\top}(z,0)=(0,0)^{\top}\,,\quad z\in\Omega\,.
\end{split}
\end{equation}
Then Proposition \ref{Prandtl Existence} shows that there exists a unique solution $(B^i_0, B^i_1)$ to \eqref{Bi-Bi+1-Prandtl equation}-\eqref{Bi-Bi+1 Prandtl BC-IC} such that
\begin{equation}\nonumber
 \langle z_1\rangle^l\pa^k_t\pa^{\alpha_1}_{z_1}\pa^{\alpha_2}_{z_2}B^i_j\in
C^0([0,T];L^2(\Omega))\quad \text{for}\quad\! k+|\alpha|\leq[\tfrac{m}{2}]-1-7i\,.
\end{equation}

The same as before, the third and fourth components of \eqref{Bi-Bi+1 equation} are ODEs for $(B^{i+1}_2, B^{i+1}_3)$. Then we can continue the process and solve all $B^j$ and $E^j$ for $j=0, 1, \cdots, N$ for any $N\in\mathbb{N}$.

\subsection{Error Terms}
We can conclude that the approximate solution $W^\eps(x,t)$ for $\eps> 0$ in \eqref{expansion} has at least the smoothness such that
\begin{equation}\label{smoothness W}
W^{\varepsilon}(x,t)\in \bigcap^{[m/2]-1-7N}_{j=0}C^j([0,T];H^{[m/2]-1-7N-j}(\Omega)).
\end{equation}
and $W^{\varepsilon}$ satisfies the equation
\begin{equation}\nonumber
(\mathcal{A}_0 - \mathcal{L}^\eps)W^\eps(x_1,x_2,t)=\varepsilon^N g^\eps_B(\tfrac{x_1}{\varepsilon},x_2,t) + \varepsilon
^{N+1}g^\eps_E(x,t)\,,
\end{equation}
for $(x,t)\in\Omega\times [0,T],$ with the boundary and initial conditions
\begin{equation}\nonumber
\begin{aligned}
W^{\varepsilon}_2-W^{\varepsilon}_3=0\,,\quad W^{\varepsilon}_0=0,
\quad -W^{\varepsilon}_1+ \sqrt{2}\alpha W^{\varepsilon}_2=0\,,\quad &\text{on}\quad\!\Gamma\times[0,T),\\
W^{\varepsilon}(x,0)=U_0(x)\,,\quad &\text{for}\quad\! x\in\Omega.
\end{aligned}
\end{equation}
where the precise expressions of $g^\eps_E(x,t)$ and $g^\eps_B(\frac{x_1}{\eps},x_2,t)$ are lengthy and not important. The smoothness and the compatibility conditions satisfied by $g^\eps_E$ and $g^\eps_B$ are
\begin{equation}\label{gE smooth}
g^\eps_E(x,t)\in \bigcap^{[m/2]+1-7N}_{j=0}C^j([0,T];H^{[m/2]+1-7N-j}(\Omega)),
\end{equation}
\begin{equation}\label{gE compatibility}
\partial^k_t g^\eps_E(x,0)=0,\;\;\;\text{for}\;x\in \Omega,\;\;k=0,1,
\end{equation}
and
\begin{equation}\label{gB smooth}
g_B(\varepsilon;\tfrac{x_1}{\varepsilon},x_2,t)\in \bigcap^{[m/2]-2-7N}_{j=0}C^j([0,T];H^{[m/2]-2-7N-j}(\Omega)),
\end{equation}
\begin{equation}\label{gB compatibility}
\pa^k_t g^\eps_B(\tfrac{x_1}{\varepsilon},x_2,0)=0\,,\quad \text{for}\quad\! x\in \Omega\quad\! k=0,1,\cdots ,[\tfrac{m}{2}]-2-7N.
\end{equation}

Moreover, it is easy to see that there exists a constant $C$ which is independent of $\varepsilon$, such that
\begin{equation}\label{gE estimate}
\sup_{t\in[0,T]}\sum_{k+|\alpha|\leq[m/2]+1-7N}\|\pa^k_t\pa^{\alpha_1}_{x_1}
\pa^{\alpha_2}_{x_2}g^\eps_E(x,t)\|_{L^2(\Omega)^2}\leq C,
\end{equation}
and
\begin{equation}\label{gB estimate}
\sup_{t\in[0,T]}\sum_{k+|\alpha|\leq[m/2]-2-7N}\|\langle z_1\rangle^l\pa^k_t
\pa^{\alpha_1}_{z_1}\pa^{\alpha_2}_{z_2}g^\eps_B(z,t)\|_{L^2(\Omega)^2}\leq C.
\end{equation}

\section{Estimates of the Error Term of the Approximate Solution}

In this section we estimate the error term of the approximate solution. Let $V^\eps$ be the solution of the linearized Navier-Stokes-Fourier equations \eqref{linearized NSF}, and $W^{\varepsilon}$ be the approximate solution we constructed in the previous sections.

Let $$w^{\varepsilon}=V^\eps -Q^{-1}(x,t)W^{\varepsilon}\doteq(w^{\varepsilon}_0,w^{\varepsilon}_1,w^{\varepsilon}_2,w^{\varepsilon}_3)^{\top}.$$
By Proposition \ref{lcns} and the \eqref{smoothness W}, we have that
\begin{equation}\label{65}
w^{\varepsilon}(x,t)\in \bigcap^{[m/2]-1-7N}_{j=0}C^j([0,T];H^{[m/2]-1-7N-j}(\Omega)).
\end{equation}
We also have that $w^{\varepsilon}$ satisfies the equation
\begin{equation}\label{w equation}
A_0(V')\pa_t w^{\varepsilon}+\sum^2_{j=1}A_j(V')\pa_j w^{\varepsilon}
-L_{\varepsilon}w^{\varepsilon}=\varepsilon^NG(\varepsilon,\tfrac{x_1}{\varepsilon},x,t)\,,\quad (x,t)\in\Omega\times[0,T].
\end{equation}
with the boundary and initial condition
\begin{equation}\label{w BC}
\begin{aligned}
w^\varepsilon_1=w^\varepsilon_2=w^\varepsilon_3&=0\,,\quad \text{on}\quad\!\Gamma\times [0,T]\\
w^\varepsilon(x,0)&=0\,,\quad\text{for}\quad\! x\in\Omega\,.
\end{aligned}
\end{equation}
Here
\begin{equation}\nonumber
\begin{split}
G(\varepsilon,\tfrac{x_1}{\varepsilon},x,t)&=\varepsilon Q^{-1}g_E(\varepsilon,x,t)
+Q^{-1}g_B(\varepsilon,\tfrac{x_1}{\varepsilon},x_2,t)\\
& =(G_0,G_1,G_2,G_3)^{\top}.
\end{split}
\end{equation}

Let $m\geq 2(7N+4)$, By \eqref{gE smooth}-\eqref{gB estimate}, we have
\begin{equation}\nonumber
G(\varepsilon;\tfrac{x_1}{\varepsilon},x,t)\in \bigcap^2_{j=0}C^j([0,T];H^{2-j}(\Omega)),
\end{equation}
and
\begin{equation}\nonumber
\pa^k_tG(\varepsilon;\tfrac{x_1}{\varepsilon},x,0)=0\,,\quad\! k=0,1,\quad\!\text{for}\quad\!x\in\Omega.
\end{equation}

Then Theorem 1.1 is a conclusion of  the following proposition:
\begin{proposition}\label{proposition1}
Assume that  $m\geq 2(7N+4)$, $w^{\varepsilon}$ satisfies \eqref{w equation}-\eqref{w BC}, then we have
\begin{equation}\nonumber
w^{\varepsilon}\in \bigcap^3_{j=0}C^j([0,T];H^{3-j}(\Omega)),
\end{equation}
and the following estimates hold
\begin{equation}\nonumber
\sup_{(x,t)\in\Omega\times[0,T]}|w^{\varepsilon}_0|\leq C\varepsilon^{N-1},
\end{equation}
and
\begin{equation}\nonumber
\sup_{(x,t)\in\Omega\times[0,T]}|w^{\varepsilon}_j|\leq C\varepsilon^{N-\frac{3}{4}},\;\;\;j=1,2,3.
\end{equation}
\end{proposition}

We write $w=(w_0,w_1,w_2,w_3)^{\top}$ instead of $w^\varepsilon$ for simplicity.
We rewrite the equation as
\begin{equation}\label{w0 equation}
\tfrac{1}{\rho'}\pa_tw_0+\tfrac{\mathrm{u}'}{\rho'}\cdot\bigtriangledown w_0 +(\pa_1w_1+\pa_2w_2)=\varepsilon^NG_0,
\end{equation}
\begin{equation}\label{w1 equation}
\begin{split}
\tfrac{\rho'}{p'_{\rho}}\pa_tw_1+\tfrac{\rho'}{p'_{\rho}}\mathrm{u}'\cdot\bigtriangledown w_1+
\pa_1w_0+\tfrac{p'_{\theta}}{p'_{\rho}}\pa_1w_3\\
-\tfrac{1}{p'_{\rho}}\varepsilon^2(\triangle w_1+C\pa_1(\pa_1w_1+\pa_2w_2))
=\varepsilon^NG_1,
\end{split}
\end{equation}
\begin{equation}\label{w2 equation}
\begin{split}
\tfrac{\rho'}{p'_{\rho}}\pa_tw_2+\tfrac{\rho'}{p'_{\rho}}\mathrm{u}'\cdot\bigtriangledown w_2+
\pa_2w_0+\tfrac{p'_{\theta}}{p'_{\rho}}\pa_2w_3\\-\tfrac{1}{p'_{\rho}}\varepsilon^2(\triangle w_2+C\pa_2(\pa_1w_1+\pa_2w_2))
=\varepsilon^NG_2,
\end{split}
\end{equation}
\begin{equation}\label{w3 equation}
\begin{split}
\beta'\pa_t w_3+\beta'\mathrm{u}'\cdot\bigtriangledown w_3+\tfrac{p'_\theta}{p'_\rho}(\pa_1w_1+\pa_2w_2)\\
-\eps^2\tfrac{\overline{\kappa}}{\theta'p'_\rho}\triangle w_3-I(w)
=\varepsilon^NG_3\,,
\end{split}
\end{equation}
where $\beta'=\frac{\rho'c_v(\theta')}{\theta'p'_\rho}$ and
\begin{equation*}
I(w)= \eps^2 \tfrac{\overline{\kappa}}{\theta'p'_\rho}
\Delta w_3\\
+\tfrac{1}{\theta'p'_\rho}(\mathbb{S}:\nabla(w_1,w_2)^{\top}+\mathbb{S}(w_1,w_2)^{\top}:\nabla \mathrm{u}')\,.
\end{equation*}
In the rest of this section, we will denote $\langle\cdot,\cdot\rangle$ as the inner product
in $L^2(\Omega)$ and $\|\cdot\|$ as the norm in  $L^2(\Omega)$, the generic constants $C_i,\,i=1,2,...$ are positive depending
only on $V'$ and its derivatives. First we derive the basic energy estimate on $w$:
\begin{lemma}\label{lemma1}
For any $t\in[0,T]$, we have that
$$\|w(t)\|^2+\varepsilon^2\sum_{j=1}^{3}\int^t_0\|\nabla w_j\|^2\,\mathrm{d}s\leq C \varepsilon^{2N+1}.$$
\end{lemma}
\begin{proof} Taking the inner product in $L^2(\Omega)$ of \eqref{w0 equation}-\eqref{w3 equation} with $w$, by
integration by parts we have that:
\begin{equation}\label{w0 norm}
\tfrac{1}{2}\tfrac{\mathrm{d}}{\mathrm{d}t}\langle\tfrac{1}{\rho'}w_0,w_0\rangle-\tfrac{1}{2}\langle
[\pa_t(\tfrac{1}{\rho'})+\nabla\cdot\tfrac{u'}{\rho'}]w_0,w_0\rangle
-\sum_{j=1}^2\langle w_j+\pa_jw_0\rangle=\langle\varepsilon^N G_0,w_0\rangle,
\end{equation}
\begin{equation}\label{w1 norm}
\begin{aligned}
&\tfrac{1}{2}\tfrac{\mathrm{d}}{\mathrm{d}t}\langle\tfrac{\rho'}{p'_{\rho}}w_1,w_1\rangle
-\tfrac{1}{2}\langle[\pa_t(\tfrac{1}{\rho'})+\nabla\cdot\tfrac{\mathrm{u}'}{\rho'}]w_1,w_1\rangle+
\langle\pa_1w_0,w_1\rangle\\
&+\langle\tfrac{p'_{\theta}}{p'_{\rho}}\pa_1w_3,w_1\rangle+\varepsilon^2\{\langle\tfrac{1}{p'_{\rho}}w_1,w_1\rangle
+C\langle\tfrac{1}{p'_{\rho}}(\pa_1w_1+\pa_2w_2),\pa_1w_1\rangle\\
&-\langle\nabla\tfrac{1}{p'_{\rho}}\cdot\nabla w_1\rangle
-C\langle\pa_1(\tfrac{1}{p'_{\rho}})(\pa_1w_1+\pa_2w_2),w_1\rangle\}=\langle\varepsilon^NG^1,w_1\rangle\,,
\end{aligned}
\end{equation}
\begin{equation}\label{w2 norm}
\begin{aligned}
&\tfrac{1}{2}\tfrac{\mathrm{d}}{\mathrm{d}t}\langle\tfrac{\rho'}{p'_{\rho}}w_2,w_2\rangle
-\tfrac{1}{2}\langle[\pa_t(\tfrac{1}{\rho'})+\nabla\cdot\tfrac{\mathrm{u}'}{\rho'}]w_2,w_2\rangle+
\langle\pa_2w_0,w_2\rangle\\
&+\langle\tfrac{p'_{\theta}}{p'_{\rho}}\pa_2w_3,w_2\rangle+\varepsilon^2\{\langle\tfrac{1}{p'_{\rho}}w_2,w_2\rangle
+C\langle\tfrac{1}{p'_{\rho}}(\pa_1w_1+\pa_2w_2),\pa_2w_2\rangle\\
&-\langle\nabla\tfrac{1}{p'_{\rho}}\cdot\nabla w_2\rangle
-C\langle\pa_2(\tfrac{1}{p'_{\rho}})(\pa_1w_1+\pa_2w_2),w_2\rangle\}=\langle\varepsilon^NG^2,w_2\rangle\,,
\end{aligned}
\end{equation}
\begin{equation}\label{w2 norm}
\begin{aligned}
&\tfrac{1}{2}\tfrac{\mathrm{d}}{\mathrm{d}t}\langle\beta' w_3,w_3\rangle-\tfrac{1}{2}\langle
[\pa_t{\beta'}+\nabla\cdot(\beta'u')] w_3,w_3\rangle-\langle\tfrac{p'_\theta}{p'_\rho}
w_1,\pa_1w_3\rangle\\
&-\langle\tfrac{p'_\theta}{p'_\rho}w_2,\pa_2w_3\rangle-\langle\nabla\tfrac{p'_\theta}
{p'_\rho}\cdot(w_1,w_2)^{\top},w_3\rangle
+\varepsilon^2\{\langle\tfrac{\kappa_0(\theta')}{\theta'p'_\rho}\nabla w_3,\nabla w_3\rangle\\
&+
\langle\nabla\tfrac{\kappa_0(\theta')}{\theta'p'_\rho} w_3,\nabla w_3\rangle\}-\langle I(w),w_3\rangle
=\varepsilon^NG_3\,.
\end{aligned}
\end{equation}
Adding the above four equations shows that
\begin{equation}\label{w energy}
\tfrac{1}{2}\tfrac{\mathrm{d}}{\mathrm{d}t}\|w\|^2_{V'}+\varepsilon^2\sum_{j=1}^3\|\nabla w_j\|^2\leq C(\|w\|^2+\varepsilon^{2N}\|G\|^2),
\end{equation}
where
$$\|w\|_{V'}=\langle\tfrac{1}{\rho'}w_0,w_0\rangle+\langle\tfrac{\rho'}{p'_{\rho}}w_1,w_1\rangle
+\langle\tfrac{\rho'}{p'_{\rho}}w_2,w_2\rangle+\langle\beta' w_3,w_3\rangle,$$
and obviously
$$C_1 \|w\|^2\leq\|w\|^2_{V'}\leq C_2 \|w\|^2.$$
Then inequality \eqref{w energy} reads that
\begin{equation}
\|w(t)\|^2+\varepsilon^2\sum_{j=1}^3\int_0^t\|\nabla w_j\|^2\leq C(\int_0^t\|w(s)\|\,\mathrm{d}s+\varepsilon^{2N}\int^t_0\|G(\varepsilon,s)\|^2\,\mathrm{d}s)\,.
\end{equation}
It is easy to check that
$$\|G(\varepsilon,s)\|^2\leq C\varepsilon\,.$$ Then by Gronwall inequality
we complete the proof of this lemma.
\end{proof}

Next, we will get some estimates of the derivatives of $w$, by the compatibility condition, one could
get that $$\pa_tw_j=0,\;\pa_2 w_j=0\,,\quad j=1,2,\quad\text{on}\quad\!\Gamma\times[0,T].$$
However, we could not get zero boundary condition for $\pa_{x_1} w_j,$ then we need to define
the tangential derivatives of $w$:
$$D^{\mathrm{tan}}w=(\pa_t w,\chi(x_1)\pa_{x_1}w,\pa_{x_2}w),$$
where $\chi(s)\in C^3([0,\infty))$ satisfies
\begin{equation*}
\begin{aligned}
\chi(0)&=0\,,\quad\chi'(0)=1,\\
\chi'(s)&\geq 0\,,\quad \text{for}\quad\! s\in(0,\infty),\\
\chi(s)&=1\,,\quad \text{for}\quad\! s\geq 1.
\end{aligned}
\end{equation*}
We have the following estimate:
\begin{lemma}\label{lemma2}
For any $ t\in[0,T]$, we have that
$$\|D^{\mathrm{tan}}w(t)\|^2+\varepsilon^2\sum_{j=1}^3\int^t_0\|\nabla D^{\mathrm{tan}}w_j(s)\|^2\,\mathrm{d}s\leq C\varepsilon^{2N-1}.$$
\end{lemma}
\begin{proof} Applying $\pa_t$ to the equation \eqref{w0 equation}-\eqref{w3 equation}, we have that:

\begin{equation}\label{a10}
\begin{split}
&\tfrac{1}{\rho'}\pa_t(\pa_tw_0)+\tfrac{\mathrm{u}'}{\rho'}\cdot\nabla (\pa_tw_0) +(\pa_1(\pa_tw_1)+\pa_2(\pa_tw_2))\\&
+\pa_t\tfrac{1}{\rho'}\pa_tw_0+\pa_t\tfrac{\mathrm{u}'}{\rho'}\cdot\nabla w_0=\varepsilon^N\pa_t G_0,
\end{split}
\end{equation}

\begin{equation}\label{a11}
\begin{split}
&\tfrac{\rho'}{p'_{\rho}}\pa_t(\pa_tw_1)+\tfrac{\rho'}{p'_{\rho}}\mathrm{u}'\cdot\nabla (\pa_tw_1)+
\pa_1(\pa_tw_0)+\tfrac{p'_{\theta}}{p'_{\rho}}\pa_1(\pa_tw_3)\\&
+\pa_t(\tfrac{\rho'}{p'_{\rho}})\pa_t w_1+\pa_t(\tfrac{\rho'}{p'_{\rho}}\mathrm{u}')\cdot \nabla w_1+\pa_t(\tfrac{p'_{\theta}}{p'_{\rho}})
\pa_1w_3\\&
-\varepsilon^2\tfrac{1}{p'_{\rho}}(\Delta (\pa_tw_1)+C\pa_1(\pa_1(\pa_tw_1)+\pa_2(\pa_tw_2)))\\&-
\varepsilon^2\pa_t\tfrac{1}{p'_{\rho}}(\Delta w_1+C\pa_1(\pa_1w_1+\pa_2w_2))
=\varepsilon^N\pa_tG_1,
\end{split}
\end{equation}

\begin{equation}\label{a12}
\begin{split}
&\tfrac{\rho'}{p'_{\rho}}\pa_t(\pa_tw_2)+\tfrac{\rho'}{p'_{\rho}}u'\cdot\nabla (\pa_tw_2)+
\pa_2(\pa_tw_0)+\tfrac{p'_{\theta}}{p'_{\rho}}\pa_2(\pa_tw_3)\\&
+\pa_t(\tfrac{\rho'}{p'_{\rho}})\pa_t w_2+\pa_t(\tfrac{\rho'}{p'_{\rho}}u')\cdot \nabla w_2+\pa_t(\tfrac{p'_{\theta}}{p'_{\rho}})
\pa_2w_3\\&
-\varepsilon^2\tfrac{1}{p'_{\rho}}(\Delta (\pa_tw_2)+C\pa_2(\pa_1(\pa_tw_1)+\pa_2(\pa_tw_2)))\\&-
\varepsilon^2\pa_t\tfrac{1}{p'_{\rho}}(\Delta w_2+C\pa_2(\pa_1w_1+\pa_2w_2))
=\varepsilon^N\pa_tG_2,
\end{split}
\end{equation}

\begin{equation}\label{a13}
\begin{split}
&\beta'\pa_t (\pa_tw_3)+\beta'\mathrm{u}'\cdot\nabla (\pa_tw_3)+\tfrac{p'_\theta}{p'_\rho}(\pa_1(\pa_tw_1)+\pa_2(\pa_tw_2))\\
&+\pa_t\beta'\pa_tw_3+\pa_t{\beta'\mathrm{u}'}\cdot\nabla w_3+\pa_t(\tfrac{p'_\theta}{p'_\rho})(\pa_1 w_1+\pa_2 w_2)\\
&-\tfrac{\varepsilon^2\kappa_0(\theta')}{\theta'p'_\rho}\Delta (\pa_tw_3)-
\pa_t\tfrac{\varepsilon^2\kappa_0(\theta')}{\theta'p'_\rho}\Delta w_3-\pa_t I(w)
=\varepsilon^N\pa_tG_3.
\end{split}
\end{equation}

Take the inner product in $L^2(\Omega)$ of the above equations with $\pa_t w$,
integrating by parts we have:
\begin{equation}\label{w t norm}
\begin{split}
&\tfrac{1}{2}\tfrac{\mathrm{d}}{\mathrm{d}t}\|\pa_t w\|^2_{U'}+\varepsilon^2\sum_{j=1}^3\langle\nabla \pa_tw_j,\nabla \pa_tw_j\rangle
+\langle\pa_t(\tfrac{\rho'}{p'_{\rho}}\mathrm{u}')\cdot w_0,\pa_t w_0\rangle\\
&\sum_{j=1}^2\langle\pa_t(\tfrac{\rho'}{p'_{\rho}}\mathrm{u}')\cdot \nabla w_j,\pa_t w_j \rangle
+\sum_{j=1}^2\langle\pa_t(\tfrac{p'_{\theta}}{p'_{\rho}}) \pa_jw_3,\pa_t w_j \rangle\\
&+\langle\pa_t(\beta'\mathrm{u}')\cdot\nabla w_3,\pa_t w_3\rangle+\langle\pa_t(\tfrac{p'_{\theta}}{p'_{\rho}})
(\pa_1 w_1+\pa_2 w_2),\pa_t w_3\rangle\\
&\leq C(\|\pa_t w\|^2+\varepsilon^{2N}\|\pa_tG\|^2).
\end{split}
\end{equation}
As $\pa_t(\frac{u'}{\rho'})=0,$ on $\Gamma\times[0,T]$, we have that
$$\|\pa_t(\tfrac{\mathrm{u}'}{\rho'})\cdot\nabla w_0\|^2\leq C \|D^{\mathrm{tan}}w\|^2.$$
Integrating \eqref{w t norm} with respect with $t$, we have that

\begin{equation}\nonumber
\begin{aligned}
&\|\pa_t w(t)\|^2+\varepsilon^2\sum_{j=1}^3\int^t_0\|\nabla \pa_tw_j(s)\|^2\,\mathrm{d}s\\
\leq & C\{\int^t_0\|\pa_t w(s)\|^2\,\mathrm{d}s+\int^t_0[\sum_{j=1}^3\|\nabla w_j((s)\|^2+\|D^{\mathrm{tan}}w(s)\|^2]\}\,\mathrm{d}s+\varepsilon^{2N}\int^t_0\|\pa_tG(s)\|^2\,\mathrm{d}s\,.
\end{aligned}
\end{equation}
by Lemma \ref{lemma1} we could have that
\begin{equation}\label{a16}
\begin{split}
\|\pa_t w(t)\|^2+\varepsilon^2\sum_{j=1}^3\int^t_0\|\nabla \pa_tw_j(s)\|^2
\leq C [\int^t_0\|D^{\mathrm{tan}}w(s)\|^2]\}\,\mathrm{d}s+\varepsilon^{2N-1}].
\end{split}
\end{equation}

Similarly, apply $\pa_2$ to the equation \eqref{w0 equation}-\eqref{w3 equation}, by energy
estimate we could have
\begin{equation}\label{a17}
\begin{split}
\|\pa_2 w(t)\|^2+\varepsilon^2\sum_{j=1}^3\int^t_0\|\nabla \pa_2w_j(s)\|^2
\leq C [\int^t_0\|D^{\mathrm{tan}}w(s)\|^2]\}ds+\varepsilon^{2N-1}].
\end{split}
\end{equation}

Apply $\chi(x_1)\pa_1$ to \eqref{w0 equation}-\eqref{w3 equation} we have that

\begin{equation}\label{a18}
\begin{split}
&\tfrac{1}{\rho'}\pa_t(\chi\pa_1w_0)+\tfrac{\mathrm{u}'}{\rho'}\cdot\nabla (\chi\pa_1w_0) +(\pa_1(\chi\pa_1w_0)+\pa_2(\chi\pa_1w_0))\\
&-\tfrac{\mathrm{u}'}{\rho'}\chi'\pa_1w_0-\chi'\pa_1w_1
+\chi\pa_1(\tfrac{1}{\rho'})\pa_tw_0+\chi|pa_1(\tfrac{\mathrm{u}'}{\rho'})\cdot\nabla  w_0\\&=\varepsilon^N\chi\pa_1G_0,
\end{split}
\end{equation}

\begin{equation}\label{a19}
\begin{split}
&\tfrac{\rho'}{p'_{\rho}}\pa_t(\chi\pa_1w_1)+\tfrac{\rho'}{p'_{\rho}}\mathrm{u}'\cdot\nabla (\chi\pa_1w_1)+
\pa_1(\chi\pa_1w_0)+\tfrac{p'_{\theta}}{p'_{\rho}}\pa_1(\chi\pa_1w_3)\\
&-\tfrac{\rho'}{p'_{\rho}}\mathrm{u}'_1\pa_1w_1-\chi'\pa_1w_0-\tfrac{p'_{\theta}}{p'_{\rho}}\chi'\pa_1w_3
+\chi\pa_1(\tfrac{\rho'}{p'_{\rho}})\pa_tw_1\\&+\chi\pa_1(\tfrac{\rho'}{p'_{\rho}}\mathrm{u}')\cdot\nabla w_1+\chi\pa_1
(\tfrac{p'_{\theta}}{p'_{\rho}})\pa_1 w_3
-\tfrac{1}{p'_{\rho}}\varepsilon^2(\Delta(\chi\pa_1w_1)\\&+C\pa_1(\pa_1(\chi\pa_1w_1)+\pa_2(\chi\pa_1w_1)))
+\tfrac{1+C}{p'_{\rho}}\varepsilon^2(\chi''\pa_1w_1+2\chi'\pa^2_1w_1)\\&+\tfrac{C}{p'_{\rho}}\varepsilon^2\chi'\pa^2_{12}w_2
-\chi\varepsilon^2\pa_1(\tfrac{1}{p'_{\rho}})(\Delta w_1+C\pa_1(\pa_1w_1+\pa_2w_2))
=\varepsilon^N\chi\pa_1G_1,
\end{split}
\end{equation}

\begin{equation}\label{a20}
\begin{split}
&\tfrac{\rho'}{p'_{\rho}}\pa_t(\chi\pa_1w_2)+\tfrac{\rho'}{p'_{\rho}}\mathrm{u}'\cdot\nabla (\chi\pa_1w_2)+
\pa_2(\chi\pa_1w_0)+\tfrac{p'_{\theta}}{p'_{\rho}}\pa_2(\chi\pa_1w_3)\\
&-\tfrac{\rho'}{p'_{\rho}}\mathrm{u}'_1\pa_1w_2-\tfrac{p'_{\theta}}{p'_{\rho}}\chi'\pa_2w_3
+\chi\pa_1(\tfrac{\rho'}{p'_{\rho}})\pa_tw_2+\chi\pa_1(\tfrac{\rho'}{p'_{\rho}}\mathrm{u}')\cdot\nabla w_2\\&+\chi\pa_1
(\tfrac{p'_{\theta}}{p'_{\rho}})\pa_2 w_3
-\tfrac{1}{p'_{\rho}}\varepsilon^2(\Delta(\chi\pa_1w_2)+C\pa_2(\pa_1(\chi\pa_1w_1)+\pa_2(\chi\pa_1w_1)))\\&
+\tfrac{1}{p'_{\rho}}\varepsilon^2(\chi''\pa_1w_2+2\chi'\pa^2_1w_2)+\tfrac{C}{p'_{\rho}}\varepsilon^2\chi'\pa^2_{12}w_1\\&
-\chi\varepsilon^2\pa_1(\tfrac{1}{p'_{\rho}})(\Delta w_2+C\pa_2(\pa_1w_1+\pa_2w_2))
=\varepsilon^N\chi\pa_1G_2,
\end{split}
\end{equation}

\begin{equation}\label{a21}
\begin{split}
&\beta'\pa_t (\chi\pa_1w_3)+\beta'u'\cdot\nabla (\chi\pa_1w_3)+\tfrac{p'_\theta}{p'_\rho}(\pa_1(\chi\pa_1w_1)+\pa_2(\chi\pa_1w_2))\\&
-\beta'u'\chi'\pa_1w_3-\tfrac{p'_\theta}{p'_\rho}\chi'\pa_1w_1+\chi\pa_1\beta'\pa_t w_3+\chi\pa_1(\beta'u')\cdot\nabla w_3\\&+
\chi\pa_1(\tfrac{p'_\theta}{p'_\rho})(\pa_1w_1+\pa_2w_2)
-\tfrac{\varepsilon^2\kappa_0(\theta')}{\theta'p'_\rho}\Delta(\chi\pa_1w_3)+\tfrac{\varepsilon^2\kappa_0(\theta')}{\theta'p'_\rho}\chi''\pa_1 w_3\\&
+\tfrac{2\varepsilon^2\kappa_0(\theta')}{\theta'p'_\rho}\chi'\pa^2_1 w_3-
\pa_1(\tfrac{\varepsilon^2\kappa_0(\theta')}{\theta'p'_\rho})\chi\Delta w_3-\chi\pa_1I(w)
=\varepsilon^N\chi\pa_1G_3,
\end{split}
\end{equation}
Since $\mathrm{u}'_1=0$ on $\Gamma\times [0,T],$ we have that
\begin{equation}\label{b1}
\langle\tfrac{\mathrm{u}'_1}{\rho'}\chi'\pa_1w_0,\chi\pa_1 w_0\rangle\leq C\|\chi\pa_1 w_0\|^2,
\end{equation}
and we also have that
\begin{equation}\label{b2}
\langle\tfrac{\chi'}{p'_\rho}\pa_{11} w_1, \chi\pa_1w_1\rangle=-\tfrac{1}{2}(\langle\pa_1(\tfrac{\chi'}{p'_\rho})
\pa_1w_1,\chi\pa_1 w_1\rangle-\langle\tfrac{\chi'}{p'_\rho}
\pa_1w_1,\chi'\pa_1 w_1\rangle),
\end{equation}
and
\begin{equation}\label{b3}
\begin{split}
&\int_{\Omega}|\chi(x_1)\pa_{x_1}g_B(\tfrac{x_1}{\varepsilon},x_2,t)|^2\,\mathrm{d}x_1\mathrm{d}x_2\\=&\varepsilon
\int_{\Omega}|\tfrac{\chi(\varepsilon z_1)}{\varepsilon}\pa_{z_1}g_B(z_1,x_2,t)|^2\,\mathrm{d}z_1\mathrm{d}x_2\\
\leq &C \varepsilon.
\end{split}
\end{equation}

Let us take the inner product in $L^2(\Omega)$ of \eqref{a18})-\eqref{a21} with $\chi\pa_1 w_1$, integrating with respect
to $t$. Then by \eqref{b1}, \eqref{b2}) and \eqref{b3} we could get that
\begin{equation}\label{a22}
\begin{split}
&\|\chi(x_1)\pa_1 w(t)\|^2+\varepsilon^2\sum_{j=1}^3\int_0^t\|\nabla(\chi\pa_1 w_j(s))\|^2\,\mathrm{d}s\\
\leq &C(\int^t_0\|D^{\mathrm{tan}}w(s)\|^2ds+\varepsilon^{2N-1}).
\end{split}
\end{equation}

Collecting the estimates \eqref{a16})-\eqref{a17} and \eqref{a22}, we complete the proof of this lemma.
\end{proof}

\begin{lemma}\label{lemma3}
For all $t\in[0,T]$, we have the following estimate
$$\sum_{j=1}^3\|\pa_1 w_j(t)\|^2\leq C\varepsilon^{2N}.$$
\end{lemma}

\begin{proof}
By estimate \eqref{w energy}, we have that
\begin{equation}\label{a23}
\begin{aligned}
\varepsilon^2\sum_{j=1}^3\|\pa_{1}w_j(t)\|^2
\leq&|\tfrac{1}{2}\tfrac{\mathrm{d}}{\mathrm{d}t}\|w(t)\|^2_{U'}|+C_2 (\|w(t)\|^2+\varepsilon^{2N}\|G\|^2)\\
\leq &C\{\varepsilon\|\pa_t w(t)\|^2+\tfrac{1}{\varepsilon}\|w(t)\|^2+\|w(t)\|^2+\varepsilon^{2N+1}\}\\
\end{aligned}
\end{equation}
by Lemma \ref{lemma1} and Lemma \ref{lemma2} we have that
$$\|w(t)\|^2\leq C \varepsilon^{2N+1},$$
and
$$\|\pa_t w(t)\|^2\leq C \varepsilon^{2N-1}\,.$$
then we complete the proof of this lemma.
\end{proof}

\begin{lemma}\label{lemma4}
For any $ t\in[0,T]$, the following estimate holds
$$\varepsilon^2\|\pa_1w_0(t)\|^2+ \eps\int_0^t\|\pa_1 w_0(s)\|^2\,\mathrm{d}s\leq C\varepsilon^{2N-1}.$$
\end{lemma}

\begin{proof}
Apply $\pa_1$ to the equation \eqref{w0 equation} and we get that
\begin{equation}\label{a24}
\begin{split}
&\tfrac{1}{\rho'}\pa_t(\pa_1 w_0)+\tfrac{\mathrm{u}'}{\rho'}\cdot\nabla (\pa_1 w_0)+\pa_1(\tfrac{1}{\rho'} \mathrm{u}')\cdot\nabla
w_0+\pa_{11}w_1+\pa_{12}w_2\\=&\varepsilon^N\pa_1 G_0,
\end{split}
\end{equation}
by \eqref{w1 equation}, we have that
\begin{equation}\label{c1}
\begin{split}
\pa_{11}w_1=&\tfrac{1}{\varepsilon^2(1+C)}(\rho'\pa_tw_1+\rho'\mathrm{u}'\cdot\nabla w_1\\&+p'_{\rho}\pa_1 w_0+p'_\theta\pa_1w_3
-C\pa_{12}w_2-p'_\rho\varepsilon^NG_1)\,.
\end{split}
\end{equation}
Thus we could eliminate $\pa_{11}w_1$ from \eqref{a24}, and take the inner product in $L^2(\Omega)$ of this equality
with $\pa_1 w_0(t)$, and integrating with respect to $t$, we obtain
\begin{equation}
\begin{split}
&\|\pa_1w_0(t)\|^2+\frac{1}{\varepsilon^2}\int^t_0\|\pa_1 w_0(s)\|^2\,\mathrm{d}s\\
\leq& C (\int^t_0\|\pa_1 w_0(s)\|^2ds+\frac{1}{\varepsilon^2}\int^t_0\sum_{j=1}^3\|D^{\mathrm{tan}}w_j(s)\|^2\,\mathrm{d}s+\varepsilon^{2N-1})
\end{split}
\end{equation}
by Lemma \ref{lemma2}, we deduce that
$$\|\pa_1 w_0(t)\|^2+\frac{1}{\varepsilon^2}\int^t_0\|\pa_1 w_0(t)\|^2 ds\leq C \varepsilon^{2N-3},$$
then we complete the proof of this lemma.
\end{proof}

Next, we will derive some estimates of higher-order derivatives of $w$.

\begin{lemma}\label{lemma5}
For any $ t\in[0,T]$, we have that
\begin{equation}\label{a25}
\begin{split}
&\|\pa_{22}w(t)\|^2+\|\chi\pa_{12}w(t)\|^2+\|\chi\pa_{11}w(t)\|^2+\int^t_0\sum_{j=1}^2\|\chi\pa_{tj}w_0(s)\|^2\,\mathrm{d}s\\
&+\varepsilon^2\sum_{j=1}^3\int^t_0(\|\nabla(\pa_{22}w_j(s))\|^2+\|\chi\nabla(\pa_{12}w_j(s))\|^2+\|\nabla(\chi\pa_{11}w_j(s))\|^2\\
&+\|\pa_{11}w_j(s)\|^2)\,\mathrm{d}s \leq C\varepsilon^{2N-3}.
\end{split}
\end{equation}
\end{lemma}

\begin{proof}
First, apply $\pa_{22}$ to the equation \eqref{w0 equation}-\eqref{w3 equation}, and
take the inner product in $L^2(\Omega)$ with $\pa_{22}w$, we could get that
\begin{equation}\label{a26}
\begin{split}
&\|\pa_{22}w(t)\|^2+\varepsilon^2\sum_{j=1}^3\int^t_0\|\nabla \pa_{22}w_j(s)\|^2\,\mathrm{d}s\\
\leq&C(\int^t_0(\|\pa_{22}w_0(s)\|^2+\|\pa_{tx_2}w_0(s)\|^2+\|\pa_2(\chi\pa_1 w_0(s))\|^2)\,\mathrm{d}s +\varepsilon^{2N-3}).
\end{split}
\end{equation}
Applying $\pa_{2}$ to the equation (\ref{a18})-(\ref{a21}), and
taking the inner product in $L^2(\Omega)$ with $\pa_{2}(\chi\pa_1w)$, we could get that
\begin{equation}\label{a27}
\begin{split}
&\|\pa_2(\chi\pa_1 w(t))\|^2+\varepsilon^2\sum_{j=1}^3\int_0^t\|\nabla\pa_2(\chi\pa_1 w_j(s))\|^2ds\\
\leq& C(\int_0^t[\|\pa_2(\chi\pa_1 w(s))\|^2+\|\pa_{22}w_0(s)\|^2+\|\chi\pa_{11}w_0(s)\|^2\\&
+\|\pa_{tx_2}w_0(s)\|^2+\|\pa_t(\chi\pa_1 w_0(s))\|^2]ds+\varepsilon^{2N-3}),
\end{split}
\end{equation}
take inner product in $L^2(\Omega)$ of (\ref{a18}) with $\chi\pa_1 w_0,$ we derive
\begin{equation}\label{a28}
\int_0^t\|\pa_t(\chi \pa_1w_0(s))\|^2\,\mathrm{d}s \leq C \{\int_0^t[\|\pa_2(\chi \pa_1w_0(s))\|^2+\|\chi\pa_{11}w_0\|^2]ds+\varepsilon^{2N-3}\}\,.
\end{equation}
By the same way, applying $\pa_2$ to \eqref{w0 equation} and taking inner product in $L^2(\Omega)$ with
$\pa_2 w_0,$ we have that
\begin{equation}\label{a29}
\int_0^t\|\pa_t( \pa_2w_0(s))\|^2\,\mathrm{d}s \leq C\{\int_0^t[\|\pa_2(\chi \pa_1w_0(s))\|^2+\|\chi\pa_{22}w_0\|^2]ds+\varepsilon^{2N-3}\}\,.
\end{equation}
Plugging  \eqref{a28} and \eqref{a29} into \eqref{a27}, we obtain that
\begin{equation}\label{a30}
\begin{split}
&\|\pa_2(\chi\pa_1 w(t))\|^2+\varepsilon^2\sum_{j=1}^3\int_0^t\|\nabla\pa_2(\chi\pa_1 w_j(s))\|^2\,\mathrm{d}s\\
\leq& C(\int_0^t[\|\pa_2(\chi\pa_1 w(s))\|^2+\|\pa_{22}w_0(s)\|^2+\|\chi\pa_{11}w_0(s)\|^2]\,\mathrm{d}s+\varepsilon^{2N-3}).
\end{split}
\end{equation}
Also by (\ref{a26}) and (\ref{a29}) we get
\begin{equation}\label{a31}
\begin{split}
&\|\pa_{22}w(t)\|^2+\varepsilon^2\sum_{j=1}^3\int^t_0\|\nabla \pa_{22}w_j(s)\|^2\,\mathrm{d}s\\
\leq&C (\int^t_0(\|\pa_{22}w_0(s)\|^2+\|\pa_2(\chi\pa_1 w_0(s))\|^2)\,\mathrm{d}s+\varepsilon^{2N-3}).
\end{split}
\end{equation}

Next, applying $\chi\pa_{11}$ to the equation \eqref{w0 equation}-\eqref{w3 equation}
and taking inner product in $L^2(\Omega)$ with $\chi\pa_{11}w,$ then we could get that
\begin{equation}\label{a32}
\begin{split}
&\|\chi\pa_{11} w(t)\|^2+\varepsilon^2\sum_{j=1}^3\int_0^t\|\nabla(\chi\pa_{11} w_j(s))\|^2\,\mathrm{d}s\\
\leq& C (\int_0^t[\|\chi\pa_{11}w(s))\|^2+\|\pa_2(\chi\pa_1w_0(s))\|^2+\varepsilon^2\sum_{j=1}^3\|\pa_{11}w_j(s)\|^2]\,\mathrm{d}s+\varepsilon^{2N-3}).
\end{split}
\end{equation}
From \eqref{c1}, we could get that
\begin{equation}\label{a33}
\begin{split}
&\int_0^t\|\pa_{11}w_1(s)\|^2\,\mathrm{d}s\\
\leq &C\{\frac{1}{\varepsilon^4}\int_0^t(\|D^{\mathrm{tan}}w_1(s)\|^2+\|\pa_1w_0(s)\|^2)\,\mathrm{d}s\\
&+\int_0^t(\|\pa_{22}w_1(s)\|^2+\|\pa_{12}w_2(s)\|^2)ds+\varepsilon^{2N-4}\int_0^t\|G(s)\|^2\,\mathrm{d}s\}\\
\leq& C\varepsilon^{2N-5},
\end{split}
\end{equation}
Similarly, we could also have
\begin{equation}\label{a34}
\begin{split}
\int_0^t\|\pa_{11}w_j(s)\|^2\,\mathrm{d}s
\leq C \varepsilon^{2N-5},\;\;\;\text{for}\quad\! j=2,3.
\end{split}
\end{equation}
Plugging \eqref{a33} and \eqref{a34} into \eqref{a32}, we obtain that
\begin{equation}\label{a35}
\begin{split}
&\|\chi\pa_{11} w(t)\|^2+\varepsilon^2\sum_{j=1}^3\int_0^t\|\nabla(\chi\pa_{11} w_j(s))\|^2\,\mathrm{d}s\\
\leq& C (\int_0^t[\|\chi\pa_{11}w(s))\|^2+\|\pa_2(\chi\pa_1w_0(s))\|^2]\,\mathrm{d}s+\varepsilon^{2N-3}).
\end{split}
\end{equation}
Collecting the estimates \eqref{a31}, \eqref{a33}, \eqref{a34} and \eqref{a35},
we finally complete the proof of this lemma.
\end{proof}

\begin{lemma}\label{lemma6}
For any $ t\in[0,T]$, we have that
$$\|\pa_{12}w_0(t)\|^2\leq C\varepsilon^{2N-5}.$$
\end{lemma}
\begin{proof}
Applying $\pa_{12}$ to \eqref{w0 equation}, taking the inner product in $L^2(\Omega)$
with $\pa_{12}w_0,$ one can show that
\begin{equation}\label{a36}
\begin{split}
\|\pa_{12}w_0(t)\|^2\leq& C \{
\int_0^t(\|\pa_{12}w_0(s)\|^2+\|\pa_{tx_1}w_0(s)\|^2\\&-\sum_{j=1}^2\langle\pa_{12}w_0,\pa_j(\pa_{12}w_j)\rangle(s))\,\mathrm{d}s
+\varepsilon^{2N-3}\}
\end{split}
\end{equation}
Applying $\pa_1$ to \eqref{w0 equation}, one could show that
\begin{equation}\label{a37}
\int_0^t\|\pa_{tx_1}w_0(s)\|^2\,\mathrm{d}s \leq C(\int_0^t\|\pa_{12}w_0(s)\|^2ds+\varepsilon^{2N-5}),
\end{equation}
Furthermore, by \eqref{c1} we have that
\begin{equation}\label{a38}
\begin{split}
&-\langle\pa_{12}w_0,\pa_1(\pa_{12}w_1)\rangle=-\langle\pa_{12}w_0,\pa_2(\pa_{11}w_1)\rangle\\
\leq&-\tfrac{C}{(1+C)\varepsilon^2}\|\pa_{12}w_0\|^2+\tfrac{C}{\varepsilon^2}
[\|\chi\pa_{12}w_1\|^2+\|\pa_{22}w_1\|^2+\|D^{\mathrm{tan}}w_1\|^2\\&+\varepsilon^{2N}\|\pa_2G_1\|^2+
\varepsilon^4\|\pa_{x_2}^3w_1\|^2+\varepsilon^4\|\pa_{x_1}\pa_{x_2}^2w_2\|^2]\,.
\end{split}
\end{equation}
Thus by Lemma \ref{lemma3}, Lemma \ref{lemma4} and Lemma \ref{lemma5} we have
\begin{equation}\label{a39}
\begin{split}
&-\int^t_0\langle\pa_{12}w_0,\pa_1(\pa_{12}w_1)\rangle(s)\,\mathrm{d}s\\
\leq&-\frac{C}{(1+C)\varepsilon^2}\int_0^t\|\pa_{12}w_0(s)\|^2\,\mathrm{d}s+C\varepsilon^{2N-5},
\end{split}
\end{equation}
we also have that
\begin{equation}\label{a40}
\begin{split}
&-\int^t_0\langle\pa_{12}w_0,\pa_1(\pa_{22}w_2)\rangle(s)ds\\
\leq& \tfrac{C}{2(1+C)\varepsilon^2}\int_0^t\|\pa_{12}w_0(s)\|^2\,\mathrm{d}s+C\varepsilon^2\int_0^t\|\pa_1(\pa_{22}w_2(s))\|^2\,\mathrm{d}s\\
\leq&\tfrac{C}{2(1+C)\varepsilon^2}\int_0^t\|\pa_{12}w_0(s)\|^2\,\mathrm{d}s+C\varepsilon^{2N-3}\,.
\end{split}
\end{equation}
Collecting the estimates  \eqref{a36}-\eqref{a40}, we have that
\begin{equation}\label{a41}
\begin{split}
\|\pa_{12}w_0(t)\|^2\leq& C\{
\int_0^t\|\pa_{12}w_0(s)\|^2\,\mathrm{d}s
+\varepsilon^{2N-5}\},
\end{split}
\end{equation}
and finally complete the proof of this lemma.
\end{proof}

\begin{lemma}\label{lemma7}
For any $ t\in[0,T]$, we have that
$$\|\pa_{tx_2}w(t)\|^2+\|\pa_{tt}w(t)\|^2\leq C \varepsilon^{2N-3}.$$
\end{lemma}
\begin{proof}
Applying $\pa_{tx_2}$ to \eqref{w0 equation}-\eqref{w3 equation}, taking the inner product in $L^2(\Omega)$
with $\pa_{tx_2}w,$ one can show that
\begin{equation}\label{a41}
\begin{split}
&\|\pa_{tx_2}w(t)\|^2+\varepsilon^2\sum_{j=1}^3\int^t_0\|\nabla\pa_{tx_2}w_j(s)\|^2\,\mathrm{d}s\\
\leq&C (\int^t_0(\|\pa_{tx_2}w(s)\|^2+\|\pa_{tt}w(s)\|^2)\,\mathrm{d}s+\varepsilon^{2N-3})
\end{split}
\end{equation}
Similarly, we have that
\begin{equation}\label{a42}
\begin{split}
&\|\pa_{tt}w(t)\|^2+\varepsilon^2\sum_{j=1}^3\int^t_0\|\nabla\pa_{tt}w_j(s)\|^2\,\mathrm{d}s\\
\leq&C (\int^t_0(\|\pa_{tx_2}w(s)\|^2+\|\pa_{tt}w(s)\|^2)\,\mathrm{d}s+\varepsilon^{2N-3})\,.
\end{split}
\end{equation}
Combining the above two inequality, we proved the lemma .
\end{proof}

\begin{lemma}\label{lemma8}
For any $ t\in[0,T]$, we have that
$$\sum_{j=1}^3\|\nabla\pa_{2}w_j(t)\|^2\leq C \varepsilon^{2N-4}.$$
\end{lemma}

\begin{proof}
Applying $\pa_{2}$ to \eqref{w0 equation}-\eqref{w3 equation}, taking the inner product in $L^2(\Omega)$
with $\pa_{2}w,$ by the above lemmas, one can show that
\begin{equation}\label{a43}
\begin{split}
\tfrac{1}{2}\tfrac{\mathrm{d}}{\mathrm{d}t}\|\pa_{2}w(t)\|^2+\varepsilon^2\sum_{j=1}^3\|\nabla\pa_{2}w_j(t)\|^2
\leq C \varepsilon^{2N-2}\,.
\end{split}
\end{equation}
Thus
\begin{equation}\label{a43}
\begin{split}
&\varepsilon^2\sum_{j=1}^3\|\nabla\pa_{2}w_j(t)\|^2\\
\leq &-\tfrac{1}{2}\tfrac{\mathrm{d}}{\mathrm{d}t}\|\pa_{2}w(t)\|^2+C\varepsilon^{2N-2}\\
\leq &\sum_{j=0}^3\|\pa_{2}w(t)\|\|\pa_{tx_2}w(t)\|+C_{44}\varepsilon^{2N-2}\\
\leq &\tfrac{1}{\varepsilon}\|\pa_{2}w(t)\|^2+\varepsilon\|\pa_{tx_2}w(t)\|^2+C \varepsilon^{2N-2}\\
\leq &C\varepsilon^{2N-2}
\end{split}
\end{equation}
\end{proof}

To get the $L^{\infty}$ norm estimate of $w$, we need the following
proposition which has been also used in \cite{xy}:
\begin{proposition}
Let $f(x_1,x_2)\in H^1(\Omega)$, and $\pa_{12}f\in L^2(\Omega).$ Then we have
$$\|f(x_1,x_2)\|_{L^{\infty}(\Omega)}\leq 2\|f\|^{\frac{1}{4}}\|\pa_1f\|^{\frac{1}{4}}\|\pa_2f\|^{\frac{1}{4}}\|\pa_{12}f\|^{\frac{1}{4}}.$$
\end{proposition}

Based on the above estimates and this proposition, we established proposition \ref{proposition1}.

\section{Linear System of Prandtl-type Equations}
In this section, we collect results on the existence on the initial boundary problems for linearized Navier-Stokes-Fourier equations,  linearized Euler equations and the linear system of Prandtl-type equations which are frequently used in the previous sections.

First, for the problem \eqref{linearized NSF} of the linearized Navier-Stokes equations
of a compressible viscous fluid for fixed $\varepsilon$, one can show the following
result by the similar argument in \cite{mn} with suitable modifications. It was also stated in \cite{xy} (see Proposition 1.3 in \cite{xy}).
\begin{proposition}\label{lcns}
Let $\varepsilon>0$ and let $m>2$ be an integer. Assume that $V_0\in H^m(\Omega)$
satisfies the compatibility condition of order $[m/2]-1$ for initial boundary problem \eqref{linearized NSF}. Then
there exist a unique solution $V$ of \eqref{linearized NSF} satisfying
$$V\in \bigcap^{[m/2]}_{j=0}C^j([0,T];H^{m-2j}(\Omega))\,.$$
\end{proposition}

We now state some result of the following linearized Euler equations with inhomogeneous source term
\begin{equation}\label{equation E}
\begin{aligned}
\mathcal{A}_0 \pa_t E - \mathcal{L} E=F\,,\quad &\text{in}\quad\!\Omega\times[0,T]\,,\\
\mathcal{M}^0 E =0\,, \quad &\text{on}\quad\!\Gamma\times[0,T]\,,\\
E(x,0)=f(x)\,,\quad &\text{for}\quad x\in\Omega\,.
\end{aligned}
\end{equation}
As explained in \cite{xy}, by modifying the argument in \cite{s},
we have the following existence of the initial value problem \eqref{equation E}:
\begin{proposition}\label{existence linarized Euler}
Let $m$ be an integer, and assume that $f(x)\in H^m(\Omega)$ and $F(x,t)\in H^m(\Omega\times [0,T])$
satisfy the compatibility condition of $m-1$. Then there exists a unique solution to \eqref{equation E}
$$E\in\bigcap^{m}_{j=0}C^j([0,T];H^{m-j}(\Omega))$$
\end{proposition}

In the rest of this section, we shall prove the property of the linear system of the Prandtl-type equations. The notation
of of the following part of this section is different from that of the other sections. The initial boundary value problem
of linear system of the Prandtl-type equations can be written as
\begin{equation}\label{1000}
\begin{split}
&a_1(0,z_2,t)\pa_t u_1+b_{11}(0,z_2,t)z_1\pa_{z_1} u_1+b_{12}(0,z_2,t) z_1\pa_{z_1}u_2\\+&c_{11}(0,z_2,t)u_1
+c_{12}(0,z_2,t)u_2
-\pa_{z_1}^2 u_1=f_1(z_1,z_2,t), \\
&a_2(0,z_2,t)\pa_t u_2+b_{21}(0,z_2,t)z_1\pa_{z_1} u_1+b_{22}(0,z_2,t) z_1\pa_{z_1}u_2\\+&c_{21}(0,z_2,t)u_1
+c_{22}(0,z_2,t)u_2-\pa_{z_1}^2 u_2=f_2(z_1,z_2,t),
\end{split}
\end{equation}
for $(z_1,z_2,t)\in \Omega\times[0,T].$
The boundary condition and initial condition are
\begin{equation}\label{1001}
\begin{split}
u_1=u_2=0\,,\quad \text{on}\quad\!\Gamma\times[0,T],\\
(u_1,u_2)(z_1,z_2,0)=(0,0)\,,\quad \text{for}\quad\! z\in\Omega\,,
\end{split}
\end{equation}
where $\Omega=\mathbb{R}^2_+=\{z=(z_1,z_2)\in\mathbb{R}^2:z_1\geq0\},\,\Gamma=\pa\Omega.$
It is assumed that
\begin{equation}\label{1002}
a_i,b_{ij},c_{ij}\in C^{\infty}(\Gamma\times[0,T]),
\end{equation}
and
\begin{equation}\label{1003}
a_i\geq c_0>0\;\;\text{on}\;\Gamma\times[0,T],
\end{equation}
for $i,j=1,2$, Let $m$ and $s$ be integers such that $m\geq s$. $(f_1,f_2)$
is assumed to satisfy the following condition:
\begin{equation}\label{1004}
\langle z_1\rangle^l\pa^k_t\pa^{\alpha_1}_{z_1}\pa^{\alpha_2}_{z_2}f_i\in C^0([0,T];L^2(\Omega))
\;\;\text{for}k+|\alpha|\leq m, \;l\in\mathbb{N}.
\end{equation}
and the compatibility condition of order $s+1$:
\begin{equation}\label{1004}
\pa^k_tf_i(z,0)=0,\;\;0\leq k\leq s.
\end{equation}

Thanks to that the coefficient of $\pa_{z_2}u$ is vanished, we could obtain that
\begin{proposition}\label{Prandtl Existence}
Assume that the conditions \eqref{1002}-\eqref{1004} hold. Then there exists a unique solution
to \eqref{1000}-\eqref{1001} satisfying
\begin{equation}\label{1005}
\langle z_1\rangle^l\pa^k_t\pa^{\alpha_1}_{z_1}\pa^{\alpha_2}_{z_2}u_i\in C^0([0,T];L^2(\Omega))
\end{equation}
for $k+|\alpha|\leq m-1,\;k+[\frac{\alpha_1+1}{2}]\leq s,$ and
$$\pa^k_t u_i(z,0)=0,\;\;k=0,1,...,\min(s,m-1),\,\,\,\;\;\;\;i=1,2$$
\end{proposition}

In order to prove Proposition \ref{Prandtl Existence}, we consider the following initial
boundary value problem for a small parameter $\delta>0$
\begin{equation}\label{p1}
\begin{split}
a(z_2,t)\pa_t \textbf{u}+b(z_2,t)z_1\pa_{z_1}\mathbf{u}+c(z_2,t)\mathbf{u}-\pa^2_{z_1}\mathbf{u}-\delta\pa^2_{z_2}\mathbf{u}
=\textbf{f}^\delta (z,t),
\end{split}
\end{equation}
$(z,t)\in \Omega\times[0,T],$ with the boundary condition
\begin{equation}\label{p2}
\begin{split}
\mathbf{u}=0\,,\quad\text{on}\quad\!\Gamma\times [0,T],\\
\mathbf{u}(z,0)=0\,,\quad\text{for}\quad\! z\in \Omega.
\end{split}
\end{equation}
where we denote
$$\mathbf{u}=(u_1,u_2)^{\top},$$
$$a(z_2,t)=\begin{pmatrix}
 a_1(z_2,t)&0\\
 0&a_2(z_2,t)
 \end{pmatrix},$$

$$b(z_2,t)=\begin{pmatrix}
 b_{11}(z_2,t)&b_{12}(z_2,t)\\
 b_{21}(z_2,t)&b_{22}(z_2,t)
 \end{pmatrix}
$$

$$c(z_2,t)=\begin{pmatrix}
 c_{11}(z_2,t)&c_{12}(z_2,t)\\
 c_{21}(z_2,t)&c_{22}(z_2,t)
 \end{pmatrix}
$$
and $\textbf{f}^\delta=(f^\delta_1,f^\delta_2)^{\top}$ which is smooth and satisfies
\begin{equation}\label{p3}
\pa_t^k \textbf{f}^\delta(z,0)=0,\;\;\text{for}\;k=0,1,...s.
\end{equation}

We could have the following uniform energy estimates, Proposition \eqref{Prandtl Existence}
is its corollary with the standard theory for linear parabolic equations and
approximate process($\delta\rightarrow 0$). In the rest part of this section, the generic constant
$C$ depend only on $a,b,c$ and are independent of $\delta$.

\begin{lemma}\label{plemma1}
The solution $\textbf{u}^\delta$ of \eqref{p1}-\eqref{p2} satisfies  that for any $ t\in[0,T],$
\begin{equation*}
\begin{split}
&\sum_{k+|\alpha|\leq m,k+[\frac{\alpha_1+1}{2}]\leq s}\|\langle z_1\rangle^l\pa^k_t\pa^{\alpha_1}_{z_1}\pa^{\alpha_2}_{z_2}\textbf{u}^\delta\|^2\\
\leq& C\sum_{k+\|\alpha|\leq m,k+[\frac{\alpha_1+1}{2}]\leq s}\int^t_0\|\langle z_1\rangle^{l+m}\pa^k_t\pa^{\alpha_1}_{z_1}\pa^{\alpha_2}_{z_2}\textbf{f}^\delta(z,s)\|^2\,\mathrm{d}s.
\end{split}
\end{equation*}
\end{lemma}

\begin{proof}

 For simplify, we shall omit the parameter $\delta$.  First, we get the $L^2$ estimate of $\textbf{u}$. Multiplying \eqref{p1} by $\langle z_1\rangle^{2l}\textbf{u}$, and integrating by part, so we obtain that
\begin{equation}\label{p4}
\begin{split}
&\frac{1}{2}\frac{\mathrm{d}}{\mathrm{d}t}\int_{\Omega}\sum_{i=1,2}\langle z_1\rangle^{2l}a_i|u_i|^2\,\mathrm{d}z-
\frac{1}{2}\int_{\Omega}\sum_{i=1,2}\langle z_1\rangle^{2l}\pa_t(a_i)|u_i|^2\,\mathrm{d}z\\&+\int_{\Omega}\sum_{i,j=1,2}\langle z_1\rangle^{2l}
b_{ij}z_1u_j\pa_{z_1}u_i\,\mathrm{d}z+\int_{\Omega}\sum_{i,j=1,2}\langle z_1\rangle^{2l}
c_{ij}u_j u_i\,\mathrm{d}z\\&+\int_{\Omega}\sum_{i=1,2}\pa_{z_1}(\langle z_1\rangle^{2l}u_i)\pa_{z_1}u_i\,\mathrm{d}z
+\int_{\Omega}\sum_{i=1,2}\delta\langle z_1\rangle^{2l}|\pa_{z_2}u_i|^2\,\mathrm{d}z\\=&\int_{\Omega}\sum_{i=1,2}\langle z_1\rangle^{2l}f_iu_i\,\mathrm{d}z,
\end{split}
\end{equation}
Note that
\begin{equation}\label{p5}
\begin{split}
&\int_{\Omega}\sum_{i,j=1,2}\langle z_1\rangle^{2l}
b_{ij}z_1u_j\pa_{z_1}u_i\,\mathrm{d}z\\=&\frac{1}{2}\int_{\Omega}\sum_{i,j=1,2}b_{ij}\pa_{z_1}(\langle z_1\rangle^{2l}z_1u_iu_j)dz-\frac{1}{2}\int_{\Omega}\sum_{i,j=1,2}b_{ij}\pa_{z_1}(\langle z_1\rangle^{2l}z_1)u_iu_j\,\mathrm{d}z\\
=&-\frac{1}{2}\int_{\Omega}\sum_{i,j=1,2}b_{ij}\pa_{z_1}(\langle z_1\rangle^{2l}z_1)u_iu_j\,\mathrm{d}z\,.
\end{split}
\end{equation}
Then we could have that for any $ t\in[0,T],$
\begin{equation}\label{p5}
\begin{split}
\sum_{i=1,2}\|\langle z_1\rangle^{l}u_i(t)\|^2+\int_0^t\sum_{i=1,2}\|\langle z_1\rangle^{l}\pa_{z_1}u_i(s)\|^2\,\mathrm{d}s\\+
\delta\int_0^t\sum_{i=1,2}\|\langle z_1\rangle^{l}\pa_{z_2}u_i(s)\|^2\,\mathrm{d}s
\leq C \int_0^t\|\langle z_1\rangle^{l}\textbf{f}^\delta(s)\|^2\,\mathrm{d}s.
\end{split}
\end{equation}

 Next, we estimate the derivative with respect to $z_2$. Applying $\pa_{z_2}^{\alpha_2}$ to (\ref{p1}),
  multiplying these equations by $\langle z_1\rangle^{2l}\pa_{z_2}^{\alpha_2}\textbf{u}$, we have that
 \begin{equation}\label{p6}
\begin{split}
&\frac{1}{2}\frac{\mathrm{d}}{\mathrm{d}t}\int_{\Omega}\sum_{i=1,2}\langle z_1\rangle^{2l}a_i|\pa_{z_2}^{\alpha_2}u_i|^2\,\mathrm{d}z-
\frac{1}{2}\int_{\Omega}\sum_{i=1,2}\langle z_1\rangle^{2l}\pa_t(a_i)|\pa_{z_2}^{\alpha_2}u_i|^2\,\mathrm{d}z\\
&-\int_{\Omega}\sum_{i,j=1,2}\pa_{z_1}(\langle z_1\rangle^{2l}z_1)
b_{ij}\pa_{z_2}^{\alpha_2}u_j\pa_{z_2}^{\alpha_2}u_i\,\mathrm{d}z+\int_{\Omega}\sum_{i,j=1,2}\langle z_1\rangle^{2l}
c_{ij}\pa_{z_2}^{\alpha_2}u_j \pa_{z_2}^{\alpha_2}u_i\,\mathrm{d}z\\&+\int_{\Omega}\sum_{i=1,2}\pa_{z_1}(\langle z_1\rangle^{2l}\pa_{z_2}^{\alpha_2}u_i)\pa_{z_1}\pa_{z_2}^{\alpha_2}u_i\,\mathrm{d}z
+\int_{\Omega}\sum_{i=1,2}\delta\langle z_1\rangle^{2l}|\pa_{z_2}\pa_{z_2}^{\alpha_2}u_i|^2\,\mathrm{d}z\\=&\int_{\Omega}\sum_{i=1,2}\langle z_1\rangle^{2l}\pa_{z_2}^{\alpha_2}f^\delta_i\pa_{z_2}^{\alpha_2}u_i\,\mathrm{d}z-\int_{\Omega}\sum_{i=1,2}\langle z_1\rangle^{2l}[\pa_{z_2}^{\alpha_2},a_i]
\pa_tu_i\pa_{z_2}^{\alpha_2}u_i\,\mathrm{d}z\\&-\int_{\Omega}\sum_{i,j=1,2}\langle z_1\rangle^{2l}[\pa_{z_2}^{\alpha_2},b_{ij}]z_1
\pa_{z_1}u_i\pa_{z_2}^{\alpha_2}u_j\,\mathrm{d}z\\&-\int_{\Omega}\sum_{i,j=1,2}\langle z_1\rangle^{2l}[\pa_{z_2}^{\alpha_2},c_{ij}]
u_i\pa_{z_2}^{\alpha_2}u_j\,\mathrm{d}z,
\end{split}
\end{equation}
For the main difficult term $\int_{\Omega}\sum_{i=1,2}\langle z_1\rangle^{2l}[\pa_{z_2}^{\alpha_2},a_i]
\pa_tu_i\pa_{z_2}^{\alpha_2}u_i$, we have that
\begin{equation}\label{p7}
\begin{split}
&\int_{\Omega}\sum_{i=1,2}\langle z_1\rangle^{2l}[\pa_{z_2}^{\alpha_2},a_i]
\pa_tu_i\pa_{z_2}^{\alpha_2}u_i\\
=&-\int_{\Omega}\sum_{i=1,2}\langle z_1\rangle^{2l}[\pa_{z_2}^{\alpha_2},a_i]
\frac{1}{a_i}\{\sum_{j=1,2}(b_{ij}z_1\pa_{z_1}u_j+c_{ij}u_j)\\&-\pa_{z_1}^2u_i-\delta\pa_{z_2}^2u_i-f_i^\delta\}\pa_{z_2}^{\alpha_2}u_i,
\end{split}
\end{equation}
and
\begin{equation}\label{p8}
\begin{split}
&|\int_{\Omega}\sum_{i=1,2}\langle z_1\rangle^{2l}[\pa_{z_2}^{\alpha_2},a_i]
\frac{1}{a_i}\pa_{z_1}^2u_i\pa_{z_2}^{\alpha_2}u_i\,\mathrm{d}z|\\
=&|-\int_{\Omega}\sum_{i=1,2}\pa_{z_1}([\pa_{z_2}^{\alpha_2},a_i])
\frac{1}{a_i}\pa_{z_1}u_i\langle z_1\rangle^{2l}\pa_{z_2}^{\alpha_2}u_i\,\mathrm{d}z-\\&\int_{\Omega}\sum_{i=1,2}[\pa_{z_2}^{\alpha_2},a_i]
\frac{1}{a_i}\pa_{z_1}u_i\pa_{z_1}(\langle z_1\rangle^{2l}\pa_{z_2}^{\alpha_2}u_i)\,\mathrm{d}z|\\
\leq &C \sum_{j\leq\alpha_2-1}\|\langle z_1\rangle^{l}\pa_{z_1}\pa_{z_2}^ju_i\|^2+\tfrac{1}{4}
\|\langle z_1\rangle^{l}\pa_{z_1}\pa_{z_2}^{\alpha_2}u_i\|^2+\|\langle z_1\rangle^{l}\pa_{z_2}^{\alpha_2}u_i\|^2,
\end{split}
\end{equation}
Then from (\ref{p6}), we could get that
\begin{equation}\label{p9}
\begin{split}
&\sum_{\alpha_2=1}^m\{\|\langle z_1\rangle^{l}\pa_{z_2}^{\alpha_2}\textbf{u}(t)\|^2+\int_0^t\|\langle z_1\rangle^{l}\pa_{z_1}\pa_{z_2}^{\alpha_2}\textbf{u}(s)\|^2\,\mathrm{d}s\}\\ \leq &C \sum_{\alpha_2=0}^m\int_0^t\|\langle z_1\rangle^{l}\pa_{z_2}^{\alpha_2}\textbf{f}^\delta(s)\|^2\,\mathrm{d}s.
\end{split}
\end{equation}

Applying $\pa^k_t$ to the equation \eqref{p1}, similarly, one could get that

\begin{equation}\label{p10}
\begin{split}
&\sup_{t\in [0,T]}\|\langle z_1\rangle^{l}\pa_t^k\textbf{u}(t)\|^2+\int_0^t\|\langle z_1\rangle^{l}\pa_{z_1}\pa_t^k\textbf{u}(s)\|^2\,\mathrm{d}s\}\\
\leq &C\{\int_0^{\top}\|\langle z_1\rangle^{l}\pa_t^k\textbf{f}^\delta(s)\|^2\,\mathrm{d}s+\sum_{k'\leq k-1}[T\sup_{t\in [0,T]}(
\|\langle z_1\rangle^{l}\pa_t^{k'}\textbf{u}(t)\|^2\\
&+\|\langle z_1\rangle^{l}\pa_{z_2}\pa_t^{k'}\textbf{u}(t)\|^2)
+\int_0^{\top}\|\langle z_1\rangle^{l}\pa_{z_1}\pa_t^{k'}\textbf{u}(s)\|^2\,\mathrm{d}s]\}\,.
\end{split}
\end{equation}
Applying $\pa^{\alpha_2}_{z_2}\pa^k_t$ to the equation \eqref{p1}, and using \eqref{p10}, similar as the estimate
\eqref{p9}, we have that
\begin{equation}\label{p11}
\begin{split}
&\sum_{k+\alpha_2\leq m,0\leq k\leq s}\|\langle z_1\rangle^l\pa^k_t\pa^{\alpha_2}_{z_2}\textbf{u}^\delta\|^2+
\sum_{k+\alpha_2\leq m,0\leq k\leq s}\int_0^t\|\langle z_1\rangle^l\pa^k_t\pa^{\alpha_2}_{z_2}\textbf{u}(s)\|^2\,\mathrm{d}s\\
\leq& C\sum_{k+\alpha_2\leq m,0\leq k\leq s}\int^t_0\|\langle z_1\rangle^{l+m}
\pa^k_t\pa^{\alpha_2}_{z_2}\textbf{f}^\delta(z,s)\|^2\,\mathrm{d}s.
\end{split}
\end{equation}

Finally, we could estimate the normal derivative. From \eqref{p11}, we have
\begin{equation}\label{p12}
\begin{split}
&\sum_{k+\alpha_2\leq m-1,0\leq k\leq s-1}\|\langle z_1\rangle^l\pa_{z_1}\pa^k_t\pa^{\alpha_2}_{z_2}\textbf{u}(t)\|^2
\\&\leq C \sum_{k+\alpha_2\leq m,0\leq k\leq s}\int^t_0\|\langle z_1\rangle^{l+m}
\pa^k_t\pa^{\alpha_2}_{z_2}\textbf{f}^\delta(z,s)\|^2\,\mathrm{d}s.
\end{split}
\end{equation}
Note that
$\pa^2_{z_1}=a(z_2,t)\pa_t \textbf{u}+b(z_2,t)z_1\pa_{z_1}\mathbf{u}+c(z_2,t)\mathbf{u}--\delta\pa^2_{z_2}\mathbf{u}
-\textbf{f}^\delta (z,t)$, iteratively, we could get that
\begin{equation}
\begin{split}
&\sum_{k+|\alpha|\leq m,k+[\frac{\alpha_1+1}{2}]\leq s}\|\langle z_1\rangle^l\pa^k_t\pa^{\alpha_1}_{z_1}\pa^{\alpha_2}_{z_2}\textbf{u}^\delta\|^2\\
\leq& C \sum_{k+\|\alpha|\leq m,k+[\frac{\alpha_1+1}{2}]\leq s}\int^t_0\|\langle z_1\rangle^{l+m}\pa^k_t\pa^{\alpha_1}_{z_1}\pa^{\alpha_2}_{z_2}\textbf{f}^\delta(z,s)\|^2\,\mathrm{d}s.
\end{split}
\end{equation}
The proof of Lemma \ref{plemma1} is complete.
\end{proof}

\noindent{\bf Acknowledgment:}
Ning Jiang was supported by a grant from the National Natural Science Foundation of China under contract  No. 11171173. Yutao Ding was supported by the postdoctoral funding of Mathematical Science Center of Tsinghua University. Ning Jiang also appreciate Prof. Z.P.Xin for his invitation of the visit to the Institute of Mathematical Sciences of CUHK between Feb-April 2012. During the visit, the conversation and suggestion of Prof. Xin play an important role in this work.

\end{document}